\documentclass[11pt]{article}
\let\epsilon\varepsilon
\usepackage[british]{babel}
\usepackage{csquotes}
\usepackage[figuresright]{rotating}

\let\phi\varphi

\usepackage{amsmath,amsthm,amssymb}
\usepackage[hidelinks,draft=false]{hyperref}
\usepackage{booktabs}
\usepackage{nicefrac,xfrac}

\usepackage{geometry}
\usepackage{enumitem}
\usepackage[all]{xy}
\usepackage{subcaption}
\usepackage{mathtools}
\usepackage{xparse}
\usepackage{floatrow}
\usepackage{caption}
\usepackage{xparse}
\usepackage{float}
\usepackage{setspace}

\usepackage{graphicx}

\usepackage{thmtools}
\declaretheorem[
style=plain,
name=Theorem,
numberwithin=section
]{thm}
\declaretheorem[
style=plain,
name=Proposition,
numberlike=thm
]{prop}
\declaretheorem[
style=plain,
name=Lemma,
numberlike=thm
]{lem}
\declaretheorem[
style=plain,
name=Lemma,
unnumbered
]{lem*}
\newtheoremstyle{TheoremNum}
        {\topsep}{\topsep}              
        {\itshape}                      
        {}                              
        {\bfseries}                     
        {.}                             
        { }                             
        {\thmname{#1}\thmnote{ \bfseries #3}}
    \theoremstyle{TheoremNum}
    \newtheorem{thmn}{Theorem}
\newtheoremstyle{PropNum}
        {\topsep}{\topsep}              
        {\itshape}                      
        {}                              
        {\bfseries}                     
        {.}                             
        { }                             
        {\thmname{#1}\thmnote{ \bfseries #3}}
    \theoremstyle{PropNum}

    \newtheoremstyle{LemNum}
        {\topsep}{\topsep}              
        {\itshape}                      
        {}                              
        {\bfseries}                     
        {.}                             
        { }                             
        {\thmname{#1}\thmnote{ \bfseries #3}}
    \theoremstyle{LemNum}
    \newtheorem{lemn}{Lemma}    
    
        \newtheoremstyle{CorNum}
        {\topsep}{\topsep}              
        {\itshape}                      
        {}                              
        {\bfseries}                     
        {.}                             
        { }                             
        {\thmname{#1}\thmnote{ \bfseries #3}}
    \theoremstyle{CorNum}
    \newtheorem{corn}{Corollary}    
    
\declaretheorem[
style=plain,
name=Corollary,
numberlike=thm
]{cor}
\declaretheorem[
style=definition,
name=Definition,
numberlike=thm
]{defn}
    
        \newtheoremstyle{DefNum}
        {\topsep}{\topsep}              
        {}                      
        {}                              
        {\bfseries}                     
        {.}                             
        { }                             
        {\thmname{#1}\thmnote{ \bfseries #3}}
    \theoremstyle{DefNum}
    \newtheorem{defnum}{Definition}

\declaretheorem[
style=definition,
name=Remark,
numberlike=thm
]{rem}

\declaretheorem[
style=plain,
name=The Mixed Littlewood Conjecture,
unnumbered
]{mconj}

\declaretheorem[
style=plain,
name=The $p$-adic Littlewood Conjecture,
unnumbered
]{pconj}

\numberwithin{equation}{section}

\newenvironment{itemize*}%
{\begin{itemize}%
	\setlength{\itemsep}{0pt}%
	\setlength{\parskip}{0pt}}%
{\end{itemize}}
\newenvironment{enumerate*}%
{\begin{enumerate}%
	\setlength{\itemsep}{0pt}%
	\setlength{\parskip}{0pt}}%
{\end{enumerate}}

\AtBeginDocument{%
}

\let\phi\varphi
\usepackage[backend=bibtex,maxbibnames=10,maxcitenames=4,maxalphanames=4,style=alphabetic,isbn=false]{biblatex}
\appto\bibsetup{\raggedright}

\begin{document}
\title{\textbf{Infinite Loops and the $p$-adic Littlewood Conjecture.\\ \Large{Part I: Reformulating the $p$-adic Littlewood Conjecture in Terms of Infinite Loops.}}}
\author{John Blackman}
\date{\today}
\maketitle

\begin{abstract} In this paper we introduce the concept of an infinite loop mod $n$ and discuss the properties that these objects have. In particular, we show that a real number $\alpha$ is a counterexample to the $p$-adic Littlewood Conjecture if and only if  there exists some $m\in\mathbb{N}$ such that $p^k\alpha$ is an infinite loop mod $p^m$, for all $k\in\mathbb{N}$. This paper is the first of a two part series, which investigate the link between infinite loops and the $p$-adic Littlewood Conjecture.
\end{abstract}

\tableofcontents

\section{Introduction}

The $p$-adic Littlewood Conjecture (pLC) is a well-known unsolved problem in Diophantine approximation. The conjecture states that for every real number $\alpha\in\mathbb{R}$, we have the following equality:
$$m_p(\alpha):=\inf\limits_{q\in\mathbb{N}}\left\{q\cdot|q|_p\cdot\|q\alpha\|\right\}=0,$$
where $|\cdot|_p$ is the \textit{$p$-adic norm} and $\|\cdot\|$ is the \textit{distance to the nearest integer function}. Of course, since $|q|_p\leq{1}$ for all $q\in\mathbb{N}$, if $\alpha$ is not an element of the \textit{set of badly approximable numbers}:
$$\textbf{Bad}:=\left\{\alpha\in\mathbb{R}\mid c(\alpha)
:=\inf\limits_{q\in\mathbb{N}}q\cdot{\|q\alpha\|}>0\right\},$$ then $\alpha$ satisfies pLC. 

With this framework in mind and with a bit of work, it is then possible to show that pLC is true if and only if the \textit{set of multiplicatively badly approximable numbers}:
$$\textbf{Mad}(p):=\left\{\alpha\in\mathbb{R}\mid \inf\limits_{k\in\mathbb{N}\cup\{0\}}c(p^k\alpha)>0\right\}$$ is empty. 
One can then use the theory of continued fractions to see that:
$$\frac{1}{B(\alpha)+2}<c(\alpha)<\frac{1}{B(\alpha)},$$
where $B(\alpha):=\sup\limits_{k\in\mathbb{N}}\{a_k:\overline{\alpha}=[a_0;a_1,\ldots]\}$ is the \textit{height function} of $\alpha$. This allows us to deduce that $\alpha\in\mathbb{R}\setminus\mathbb{Q}$ satisfies pLC if and only if:
$$\sup\limits_{\ell\in\mathbb{N}\cup\{0\}}B(p^\ell\alpha)=\infty.$$ In particular, understanding the behaviour of continued fractions under integer multiplication is very closely related to understanding the $p$-adic Littlewood Conjecture.

\subsection{Main Results}

This paper provides a novel way of looking at the $p$-adic Littlewood Conjecture, by using infinite loops. We initially derive these objects using the geometric link between cutting sequences and continued fractions, however, they can also be described as real numbers which have continued fraction expansions that satisfy certain properties. In particular:

\begin{defnum}[\ref{infloopdef} (b)] An \textit{infinite loop} mod $n$ is any real number $\alpha\in\mathbb{R}_{>0}$ with no semi-convergent denominators which are  divisible by $n$ (other than $q_{-1}=0$). (See definition~\ref{semiconv} for the formally definition of a semi-convergent).
\end{defnum}


As we will discuss later, the continued fraction expansions of infinite loops (mod $n$) behave badly under multiplication by $n$. Since the behaviour of continued fraction expansions under integer multiplication is very closely tied to the $p$-adic Littlewood Conjecture, infinite loops intuitively seem like a good place to investigate for potential counterexamples to pLC. In fact, this intuition is correct and leads to the main theorem of this paper.

\begin{thmn}[\ref{theorem}]
Let $\alpha\in\textbf{Bad}$. Then
$\alpha$ satisfies pLC if and only if there is a sequence of natural numbers $\left\{\ell_m\right\}_{m\in\mathbb{N}}$ such that $p^{\ell_m}\alpha$ is not an infinite loop mod $p^m$.
\end{thmn}

This theorem comes from the combination of two smaller results. Firstly, we show that if $\alpha$ is not an infinite loop mod $n$, then the height function $B(\cdot)$ can not be small for both $\alpha$ and $n\alpha$ simultaneously.

\begin{lemn}[\ref{noloop}]
Assume that $\alpha\in\mathbb{R}_{>0}$ is not an infinite loop mod $n$. Then:
$$\max\left\{B(\alpha),B(n\alpha)\right\}\geq{\left\lfloor{2\sqrt{n}}\right\rfloor{-1}},$$
where $\lfloor\cdot\rfloor$ is the standard floor function.
\end{lemn}

When we consider $n=p^m$ to be some prime power, this then leads an important corollary which effectively proves one direction of Theorem~\ref{theorem}:

\begin{corn}[\ref{infl}] If $\alpha\in\mathbb{R}_{>0}$ is not an infinite loop mod $p^m$, then:
\[ m_p(\alpha)\leq{\frac{1}{\left\lfloor{2\sqrt{p^m}}\right\rfloor{-1}}}.\] 
\end{corn}

Secondly, if there is some fixed $m\in\mathbb{N}$ such that $p^\ell\alpha$ is an infinite loop $p^m$ for all $\ell\in\mathbb{N}$, then we can guarantee that $B(\alpha)<p^m-4$. We can then apply this statement to $p^\ell\alpha$ - for each $\ell\in\mathbb{N}\cup\{0\}$ - to get the following result:

\begin{lemn}[\ref{count}]
Let $\alpha\in\textbf{Bad}$ and assume there exists an $m\in\mathbb{N}$ such that $p^\ell\alpha$ is an infinite loop mod $p^m$, for all $\ell\in\mathbb{N}\cup\left\{0\right\}$. Then $\alpha$ is a counterexample to pLC and $m_p(\alpha)\geq{\frac{1}{p^m-2}}$.
\end{lemn}

This proves Theorem~\ref{theorem} in the other direction.

\subsection{Structure of the Paper}
This paper is the first of a two part series, which investigate the link between infinite loops and the $p$-adic Littlewood Conjecture. This paper is more theoretical in nature and looks at the main properties of infinite loops. The second paper \cite{part2} looks at improving known upper bounds to pLC. In particular, it uses the results in this paper to construct an algorithm that provides us with better upper bounds on $m_{PLC}(p):=\sup\limits_{\alpha\in\mathbb{R}}\{m_p(\alpha)\}$.

Section~2 of this paper provides context for the rest of the paper. We begin Section~\ref{CF} by giving a brief summary of the main results regarding continued fractions that we will use. More details about these topics can be found in \cite{Khin:1963,HW:1938,Bur:2000}. Section~\ref{bad} introduces the set of badly approximable numbers, partly to provide a certain amount of context to the $p$-adic Littlewood Conjecture and partly to allow us to make further statements about the $p$-adic Littlewood Conjecture. Section~\ref{mlc2} then formally introduces the mixed and $p$-adic Littlewood Conjectures, as well as reformulating the $p$-adic Littlewood Conjecture in terms of the height function $B(\alpha)$. For an excellent overview of recent results relating to the Littlewood-type problems see \cite{Bug:14}.

Section~\ref{CS} then introduces the notion of a cutting sequence and explains the relationship between cutting sequences and continued fractions. This connection between cutting sequences and continued fractions was first noted by Humbert \cite{Humbert:1916} and built upon by Series \cite{Series:1985,Series2:1985}. Most importantly, we discuss how if $\zeta_\alpha$ is a geodesic ray in $\mathbb{H}$ starting at the $y$-axis and terminating at a point $\alpha\in\mathbb{R}_{>0}$, then the cutting sequence of $(\zeta_\alpha,\mathcal{F})$ relative to the Farey tessellation is in some way equivalent to the continued fraction expansion of $\alpha$. 

In Section~\ref{Infloop}, we will introduce the concept of an infinite loop and motivate why these objects are important. We begin with Section~\ref{mult}, which is based on the author's previous work \cite{paper} and gives an overview of how replacing the Farey tessellation $\mathcal{F}$ with the \textit{$\frac{1}{n}$-scaled Farey tessellation} $\frac{1}{n}\mathcal{F}$ induces integer multiplication by $n$ on the corresponding continued fraction expansions. We then look at the common structure of $\mathcal{F}$ and $\frac{1}{n}\mathcal{F}$, with the intention of ascertaining more information about how integer multiplication affects continued fractions. This provides the motivation for why we look at infinite loops mod $n$ and allows us to discuss some of the more important properties of infinite loops. Most importantly, we show that for every $n\geq{4}$ there exist infinite loops mod $n$. Finally, in Section~\ref{ILpLC2} we introduce and prove the main results of this paper.

This paper was adapted from the author's Ph.D. thesis \cite{thesis}. 

\section{Context}
%

\subsection{\texorpdfstring{The Mixed and $p$-adic Littlewood Conjectures}{The Mixed and p-adic Littlewood Conjectures}}\label{mlc}
\subsubsection{Continued Fractions}\label{CF}
In this Section, we will give a brief overview of some properties of continued fractions. For a more in depth look at these topics, see~\cite{Khin:1963,HW:1938,Bur:2000}.

\begin{defn}
A (simple) \textit{continued fraction} $\overline{\alpha}$ is an expression of the form \begin{singlespace}
$$\overline{\alpha}:=a_0+\cfrac{1}{a_1+\cfrac{1}{a_2+\cfrac{1}{{\ldots}}}},$$
\end{singlespace}

\noindent
where $a_0\in\mathbb{Z}$ and $a_i\in\mathbb{N}$ for $i\geq1$.
\end{defn}

We will usually identify continued fractions with their sequence of $a_i$'s, $\overline{\alpha}=[a_0;a_1,\ldots]$ and refer to the $a_i$'s as \textit{partial quotients}. We note that explicit evaluation of the continued fraction expansion produces a real number $\alpha$. Similarly, for any real number $\alpha$ we can find an associated continued fraction expansion $\overline{\alpha}$ by using Euclid's algorithm. If $\alpha\in\mathbb{Q}$, then there are two equivalent continued fraction expansions which are both finite. These two continued fraction expansions will be of the form $[a_0;a_1,\ldots,a_m,1]$ and $[a_0;a_1,\ldots,a_m+1]$. If $\alpha\in\mathbb{R}\setminus\mathbb{Q}$, then there is a unique continued fraction expansion, which has infinitely many partial quotients. 

\begin{defn}Let $\overline\alpha=[a_0;a_1,a_2,\ldots]$ be a continued fraction. We define the \textit{$k$-th convergent} of $\overline{\alpha}$ to be $\frac{p_k}{q_k}:=[a_0;a_1,\ldots,a_k]$. 
We can define this iteratively where:
\begin{align*} p_{-1} &= 1   &p_0 &= a_0  &p_k &= a_kp_{k-1}+p_{k-2} \\
q_{-1} &= 0  &q_0 &= 1  &q_k &= a_kq_{k-1}+q_{k-2} 
\end{align*}
We refer to the term $p_k$ as the \textit{$k$-th convergent numerator} of $\alpha$ and $q_k$ as the \textit{$k$-th convergent denominator}.
\end{defn}

One nice property of the convergents $\frac{p_k}{q_k}$ of a real number $\alpha$ is that the convergents give very good rational approximations of $\alpha$ (especially compared to the size of their denominator). In particular, the following Theorem gives an idea of just how good these approximations are.

\begin{thm}\label{uplow}
Let $\alpha\in\mathbb{R}$ and let $\overline{\alpha}:=[a_0;a_1,\ldots]$ be the corresponding continued fraction expansion.
Then:
$$\frac{1}{(a_{k+1}+2)q_k^2}<\left|\alpha-\frac{p_k}{q_k}\right|<\frac{1}{a_{k+1}q_k^2}.$$
\end{thm}

For our purposes, it will be useful to also introduce a slightly weaker notion of a convergent, known as a \textit{semi-convergent} or \textit{secondary convergent}.

\begin{defn}\label{semiconv}
Let $\overline\alpha=[a_0;a_1,a_2,\ldots]$ be a continued fraction expansion of some real number $\alpha$. We define the \textit{$\left\{k,m\right\}$-th semi-convergent} of $\overline{\alpha}$ to be $\frac{p_{\left\{k,m\right\}}}{q_{\left\{k,m\right\}}}:=[a_0;a_1,\ldots,a_{k},m]$, where $0\leq{m}\leq{a_{k+1}}$. 
We can define this iteratively using the standard convergents: $$ p_{\left\{k,m\right\}} = mp_{k}+p_{k-1},$$ $$
q_{\left\{k,m\right\}}  = mq_{k}+q_{k-1}.
$$
We refer to the term $p_{\left\{k,m\right\}}$ as the \textit{$\left\{k,m\right\}$-th semi-convergent numerator} of $\alpha$ and $q_{\left\{k,m\right\}}$ as the \textit{$\left\{k,m\right\}$-th semi-convergent denominator}. 
\end{defn}

\subsubsection{The Badly Approximable Numbers}\label{bad}

In order to properly discuss the mixed and $p$-adic Littlewood conjectures, we will give a very brief overview of  \textit{set of badly approximable numbers} \textbf{Bad}. More information can be found in \cite{Bur:2000}.

A real number $\alpha$ is said to be \textit{badly approximable}, if there exists some constant $c>0$ such that for every rational number $\frac{p}{q}\in\mathbb{Q}$, we have:
$$\frac{c}{q^2}\leq\left|\alpha-\frac{p}{q}\right|.$$
In other words, no rational number $\frac{p}{q}$ gives an arbitrarily good approximation of $\alpha$ relative to the size of the denominator squared $q^2$. Of course, every rational number gives an arbitrarily ``good'' approximation of itself, and so all rational numbers are well approximable.

By rearranging the above equation and taking the infimum, we can define the function $c(\alpha)$:
$$c(\alpha):=\inf\limits_{\frac{p}{q}\in\mathbb{Q}}\left\{q^2\cdot\left|\alpha-\frac{p}{q}\right|\right\}.$$ 
It follows by definition, that $\alpha$ is badly approximable if and only if $c(\alpha)>0$. As a result, we can define the set of badly approximable numbers as: 
$$\textbf{Bad}:=\{\alpha\in\mathbb{R}:c(\alpha)>0\}.$$

If we define the \textit{distance to the nearest integer function} $\|\cdot\|:\mathbb{R}\to\big[0,\frac{1}{2}\big)$ to be the function given by: 
$$\|\alpha\|:=\min\limits_{n\in\mathbb{Z}}\{|\alpha-n|\},$$
then it follows that we can rewrite $c(\alpha)$ as: $$c(\alpha):=\inf\limits_{q\in\mathbb{N}}\{q\cdot\|q\alpha\|\}.$$ 

 We define the \textit{height function} $B(\alpha)$ to be the largest partial quotient in the corresponding continued fraction expansion $\overline{\alpha}:=[a_0;a_1,\ldots]$, excluding the initial partial quotient $a_0$. In particular,
$$B(\alpha):=\sup\limits_{k\in\mathbb{N}}\{a_k:\overline{\alpha}:=[a_0;a_1,\ldots]\}.$$ Then, by noting that $q\cdot\|q\alpha\|$ is minimised when $q$ is a convergent denominator of $\alpha$ and using Theorem~\ref{uplow}, one can deduce the following: 

\begin{cor}\label{uplowcor}
For every $\alpha\in\mathbb{R}\setminus{\mathbb{Q}}$,  we have:$$\frac{1}{B(\alpha)+2}<c(\alpha)<\frac{1}{B(\alpha)}.$$
\end{cor}

In particular, one can also define the set of badly approximable numbers as:
$$\textbf{Bad}:=\{\alpha\in\mathbb{R}\setminus\mathbb{Q}:B(\alpha)<\infty\}$$
%

\subsubsection{The Mixed and $p$-adic Littlewood Conjectures}\label{mlc2}

The \textit{mixed Littlewood Conjecture} (mLC) was first proposed by de Mathan and Teuli\'e in 2004, as a 1-dimensional analogue of the classical Littlewood Conjecture \cite{dMT:2004}. The main purpose of this conjecture was to gain further insight into the Littlewood Conjecture. However, this problem has proved very interesting in its own right, and whilst significant progress has been made, the conjecture remains open. More recently, the \textit{$t$-adic Littlewood Conjecture} - an analogue of pLC over function fields - was proven to be false for $\mathbb{F}_3$ \cite{ANL:2018}. This provides some credence to the notion that pLC (or mLC) may also be false.

In order to explicitly state the mixed (and $p$-adic) Littlewood Conjecture, we must first introduce some definitions. Let $\mathcal{C}=(c_k)_{k\in\mathbb{N}}$ be a sequence of integers with $c_k\geq{2}$ for all $k$. Then we set $d_0=1$ and $d_k=c_kd_{k-1}$ for all $k\in\mathbb{N}$, \textit{i.e.} $d_k=c_1\cdot{c_2}\cdot\ldots\cdot{c_k}$. We refer to any sequence $\mathcal{D}:=(d_k)_{k\in\mathbb{N}}$ which can be defined in this way as a \textit{pseudo-absolute sequence}. If we define $v_\mathcal{D}(q):=\sup\limits_{n\in\mathbb{N}}\left\{d_n:d_n\mid{q}\right\}$, then the $\mathcal{D}\textit{-adic norm}$ (or \textit{pseudo-absolute norm}) is given by:
 $$|q|_\mathcal{D}:=\frac{1}{v_\mathcal{D}(q)}.$$

The mixed Littlewood Conjecture is then stated as follows:

\begin{mconj}
For every real number $\alpha\in\mathbb{R}$ and every pseudo-absolute sequence $\mathcal{D}$, we have:
$$m_{\mathcal{D}}(\alpha):=\inf\limits_{q\in\mathbb{N}}\left\{q\cdot|q|_\mathcal{D}\cdot\|q\alpha\|\right\}=0.$$
\end{mconj} 

Here, we note that the function $m_\mathcal{D}(\alpha)$ looks remarkably similar to the function $$c(\alpha):=\inf\limits_{q\in\mathbb{N}}\{q\cdot\|q\alpha\|\}.$$ Of course, this is no coincidence and with a bit of work one can show that:
$$m_{\mathcal{D}}(\alpha)=\inf\limits_{k\in\mathbb{N}}\{c(d_k\alpha)\},$$
where $\mathcal{D}=\{d_k\}_{k\in\mathbb{N}}$ is a pseudo-absolute norm. It is for this reason, that the set of counterexamples to mLC are occasionally referred to as \textit{the set of multiplicatively badly approximable numbers}. We will denote this set  as: $$\textbf{Mad}(\mathcal{D}):=\left\{\alpha\in\mathbb{R}: \inf\limits_{k\in\mathbb{N}}\{c(d_k\alpha)\}>0\right\}.$$ 
 
When $\mathcal{C}$ is the constant sequence (\textit{i.e.} every $c_k=a$ for some $a\geq{2}$), then we will write $|\cdot|_a$ to mean $|\cdot|_\mathcal{D}$, where $\mathcal{D}$ is the corresponding pseudo-absolute sequence. In this case, we have $\mathcal{D}=\left\{1,a,a^2,a^3,\ldots\right\}$. When $a=p$ is a prime, the $\mathcal{D}$-adic norm $|\cdot|_p$ is just the standard \textit{$p$-adic norm}. For a fixed prime $p$, we obtain a specific case of the mixed Littlewood conjecture, known as the $p$\textit{-adic Littlewood Conjecture} (pLC). We state the conjecture as follows:

\begin{pconj}
For every real number $\alpha\in\mathbb{R}$, we have:
$$m_p(\alpha):=\inf\limits_{q\in\mathbb{N}}\left\{q\cdot|q|_p\cdot\|q\alpha\|\right\}=0.$$
\end{pconj}

Similar to the case of the mixed Littlewood conjecture, we can also rewrite the $p$-adic Littlewood conjecture in terms of the function $c(\alpha)$. More specifically, $$m_p(\alpha)=\inf\limits_{\ell\in\mathbb{N}\cup\{0\}}\{c(p^\ell\alpha)\}.$$

Here, we can use Corollary~\ref{uplowcor} to get nice bounds on $m_p(\alpha)$, as well as a reformulation of pLC:

\begin{cor}\label{uplow2}
Every real number $\alpha\in\mathbb{R}\setminus\mathbb{Q}$ satisfies the following inequality:
$$\inf_{\ell\in\mathbb{N}}\frac{1}{B(p^\ell\alpha)+2}<m_{p}(\alpha)<\inf_{\ell\in\mathbb{N}}\frac{1}{B(p^\ell\alpha)}.$$
In particular, if $\alpha\in\mathbb{R}\setminus\mathbb{Q}$, then $\alpha$ satisfies pLC if and only if:
$$\sup\limits_{\ell\in\mathbb{N}\cup\{0\}}B(p^\ell\alpha)=\infty.$$
\end{cor}

Finally, we should note that if $\alpha\in[0,1]$, then $\alpha$ satisfies pLC if and only if $\alpha+k$ satisfies pLC for all $k\in\mathbb{Z}$. This follows since: 
\begin{align*}
m_p(\alpha+k) &= \inf\limits_{q\in\mathbb{N}}\left\{q\cdot|q|_p\cdot\|q(\alpha+k)\|\right\}\\
&= \inf\limits_{q\in\mathbb{N}}\left\{q\cdot|q|_p\cdot\|q\alpha+qk\|\right\}\\
&=\inf\limits_{q\in\mathbb{N}}\left\{q\cdot|q|_p\cdot\|q\alpha\|\right\}\\
&= m_p(\alpha).
\end{align*}
As a result, we will typically only look at pLC for $\alpha\in\mathbb{R}_{\geq{}0}$ (or even $\alpha\in[0,1]$). However, our results hold for all $\alpha\in\mathbb{R}$ by extension.
\subsection{Continued Fractions and Cutting Sequences of Geodesic Rays} \label{CS}
In this section, we will introduce the notion of a cutting sequence and explain the connection between cutting sequences and continued fractions. We base this section on the work of Series \cite{Series:1985,Series2:1985}.

Throughout this paper we will work with the \textit{hyperbolic plane} $\mathbb{H}$. We will represent the hyperbolic plane by the upper half plane model $\mathbb{H}:=\{z\in\mathbb{C}\cup\{\infty\}:Im(z)\geq0\}$ with boundary $\partial\mathbb{H}=\mathbb{R}\cup\{\infty\}$. Geodesic lines are given by Euclidean half-lines of the form $\{a+iy:0\leq{y}\leq\infty\}$ and semicircles centred on $\partial{\mathbb{H}}$. 

We define a \textit{hyperbolic $n$-gon} $P$ to be the region enclosed by (and including) the \textit{edges} $l_1,\ldots,l_n$, where:
\begin{enumerate}
\item each $l_i$ is a geodesic segment,
\item consecutive edges $l_i$ and $l_{i+1}$  intersect only at a common endpoint $v_i$ and no other edges pass through $v_i$ - here, we treat $l_{n+1}$ as $l_1$, 
\item and the edges are otherwise pairwise disjoint, \textit{i.e.}:$$ 
l_i\cap{l_j}= \begin{cases} \emptyset &
\text{If } j\neq{i-1,i+1} \text{,} \\ v_{i-1} \text{ or } v_i &
\text{otherwise.}  
\end{cases}$$
\end{enumerate}  

Given two consecutive edges $l_i$ and $l_{i+1}$ in $P$, we refer to the common endpoint of these edges $v_i$ as a \textit{vertex} of $P$. A hyperbolic $n$-gon is \textit{ideal} if all of its vertices lie on the boundary of the hyperbolic plane  $\partial\mathbb{H}$. A \textit{tessellation} of $\mathbb{H}$ will be a collection of hyperbolic polygons $\mathcal{P}=\{\tau_i\}_{i\in\mathbb{N}}$ such that the collection of these polygons cover $\mathbb{H}$, \textit{i.e.} $\bigcup\limits_{i\in\mathbb{N}}\tau_i=\mathbb{H}$, and for any two polygons $\tau_j,\tau_k$ in $\mathcal{P}$ these polygons either do not intersect, \textit{i.e.} $\tau_j\cap\tau_k=\emptyset$, intersect only at a common vertex, \textit{i.e.} $\tau_j\cap\tau_j=z_i$, or intersect along a common edge, \textit{i.e.} $\tau_j\cap\tau_k=l_i$, where $l_i$ is an edge of both $\tau_1$ and $\tau_2$. If $E$ is an edge of a polygon $\tau\in\mathcal{P}$, we will say that $E$ is an edge of the tessellation $\mathcal{P}$. If these polygons in $\mathcal{P}$ are all ideal $3$-gons, then we refer to $\mathcal{P}$ as an \textit{ideal triangulation} of $\mathbb{H}$. 


\subsubsection{Cutting Sequences}

Let $\zeta$ be an oriented geodesic ray which enters an ideal triangle $\bigtriangleup{ABC}$, labelled clockwise, through the edge $AB$. Then $\zeta$ can leave the triangle $\bigtriangleup{ABC}$ in one of three ways:
\begin{enumerate}
\item The geodesic $\zeta$ passes through the edge $BC$. This isolates the vertex $B$ (lying to the left of $\zeta$) from the vertices $A$ and $C$ (which lie to the right of $\zeta$). In this case, we say that $\zeta$ cuts $\bigtriangleup{ABC}$ to form a \textit{left triangle}. See Fig.~\ref{fans} (a).
\item The geodesic $\zeta$ passes through the edge $AC$. This isolates the vertex $A$ (lying to the right of $\zeta$) from the vertices $B$ and $C$ (which lie to the left of $\zeta$). In this case, we say that $\zeta$ cuts $\bigtriangleup{ABC}$ to form a \textit{right triangle}. See Fig.~\ref{fans} (b).
\item The geodesic terminates at the vertex $C$. Here, we refer to the vertex $C$ as the \textit{opposing vertex}.
\end{enumerate}

\begin{figure}[hbt]
        \centering
        \begin{subfigure}[b]{0.4\textwidth}
            \centering
            \includegraphics[width=\textwidth]{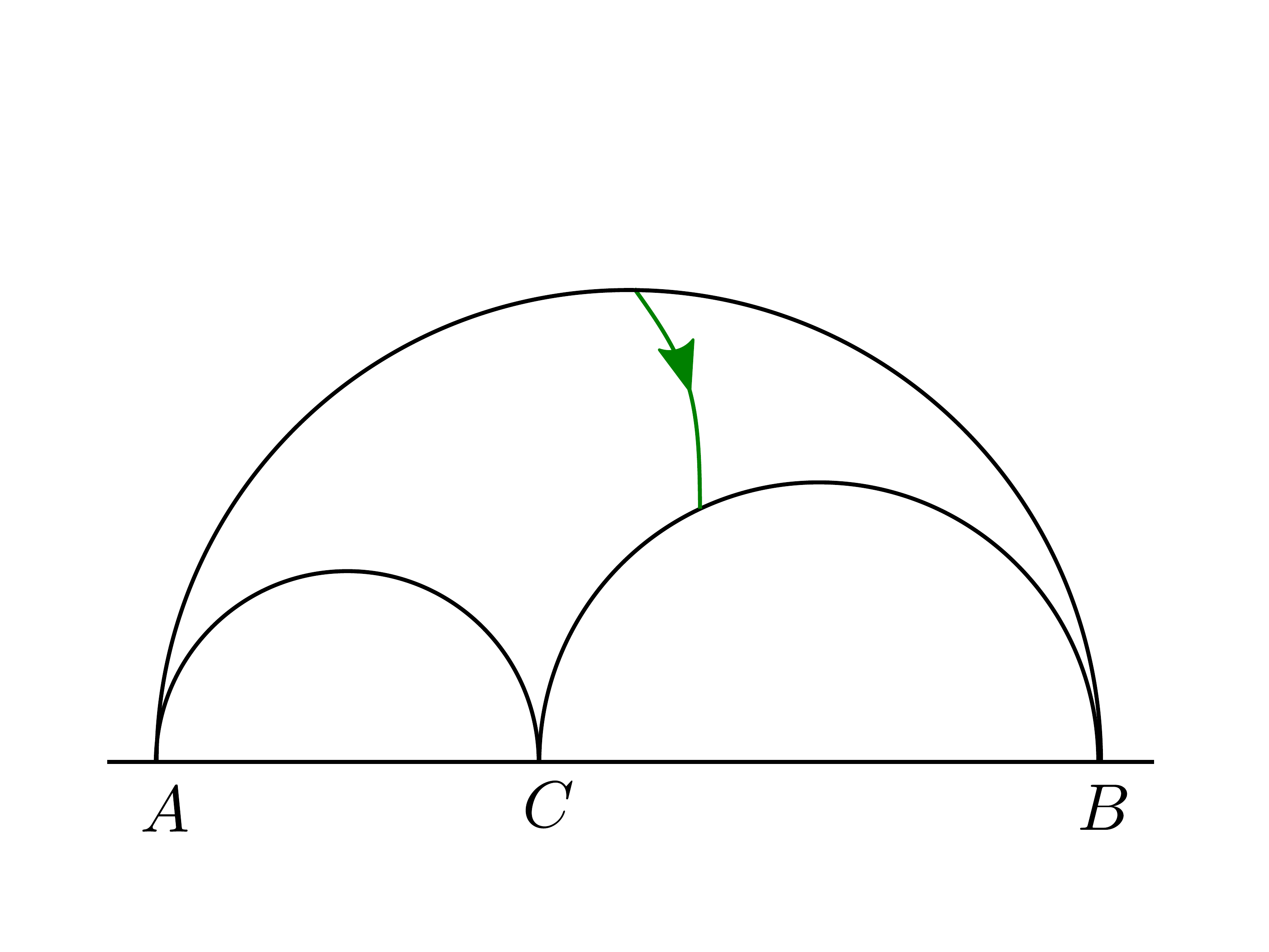}
            \caption{An example of a left triangle.}    
            
        \end{subfigure}
        \quad
        \begin{subfigure}[b]{0.4\textwidth}  
            \centering 
            \includegraphics[width=\textwidth]{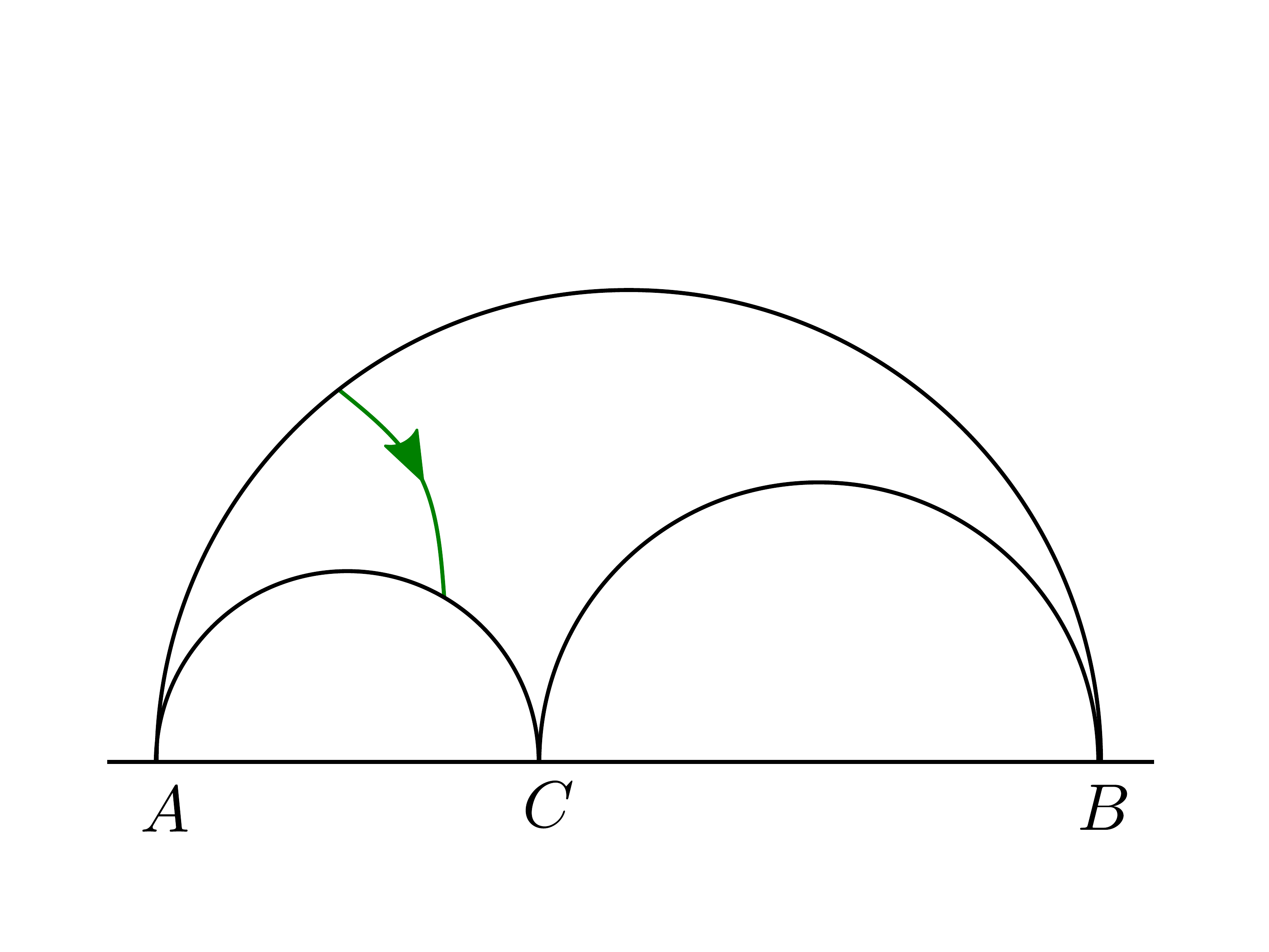}
            \caption{An example of a right triangle.}    
         
        \end{subfigure}
        \\
        \begin{subfigure}[b]{0.4\textwidth}   
            \centering 
            \includegraphics[width=\textwidth]{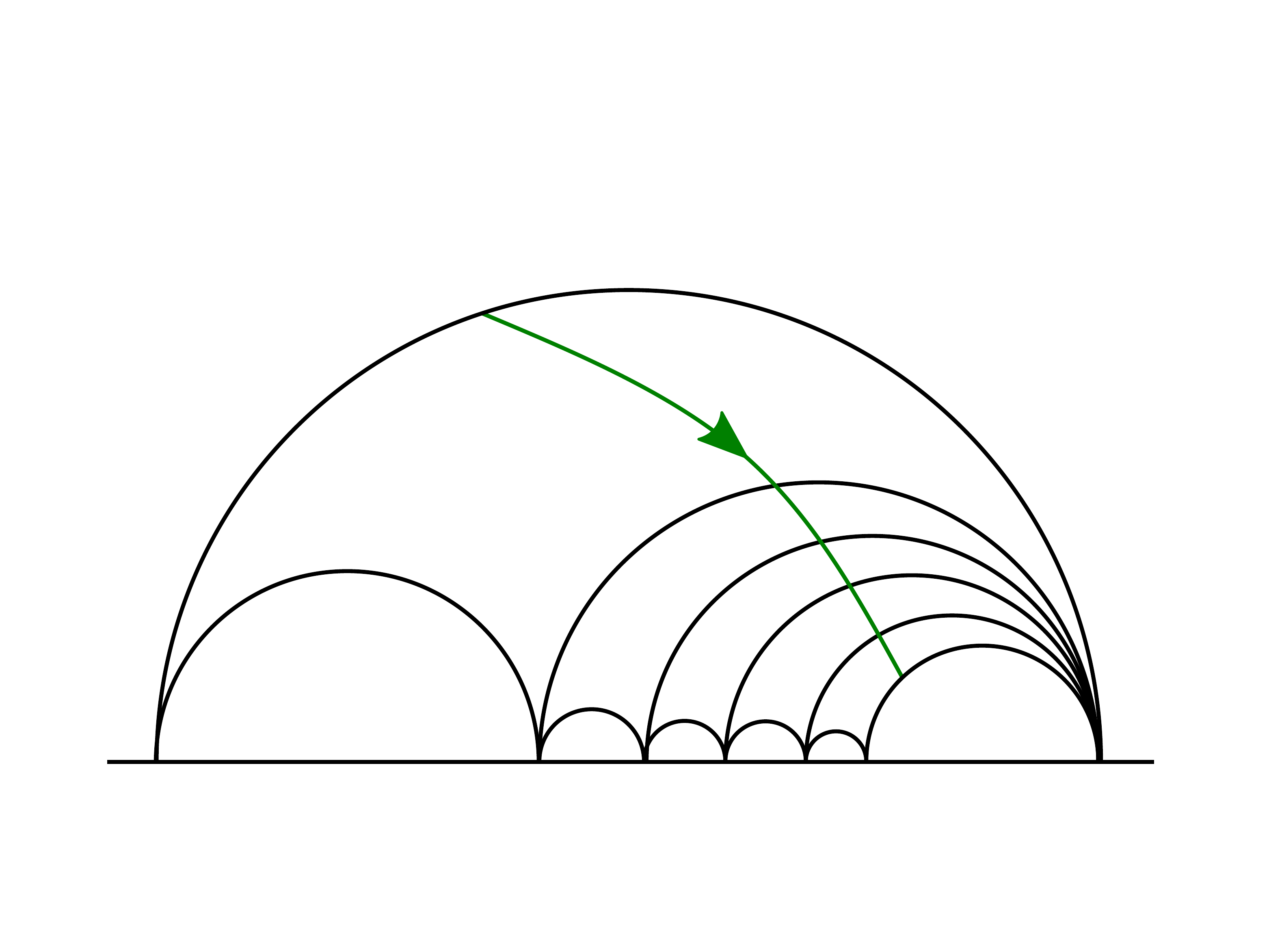}
            \caption{An example of a left fan.}    

        \end{subfigure}
        \quad
        \begin{subfigure}[b]{0.4\textwidth}   
            \centering 
            \includegraphics[width=\textwidth]{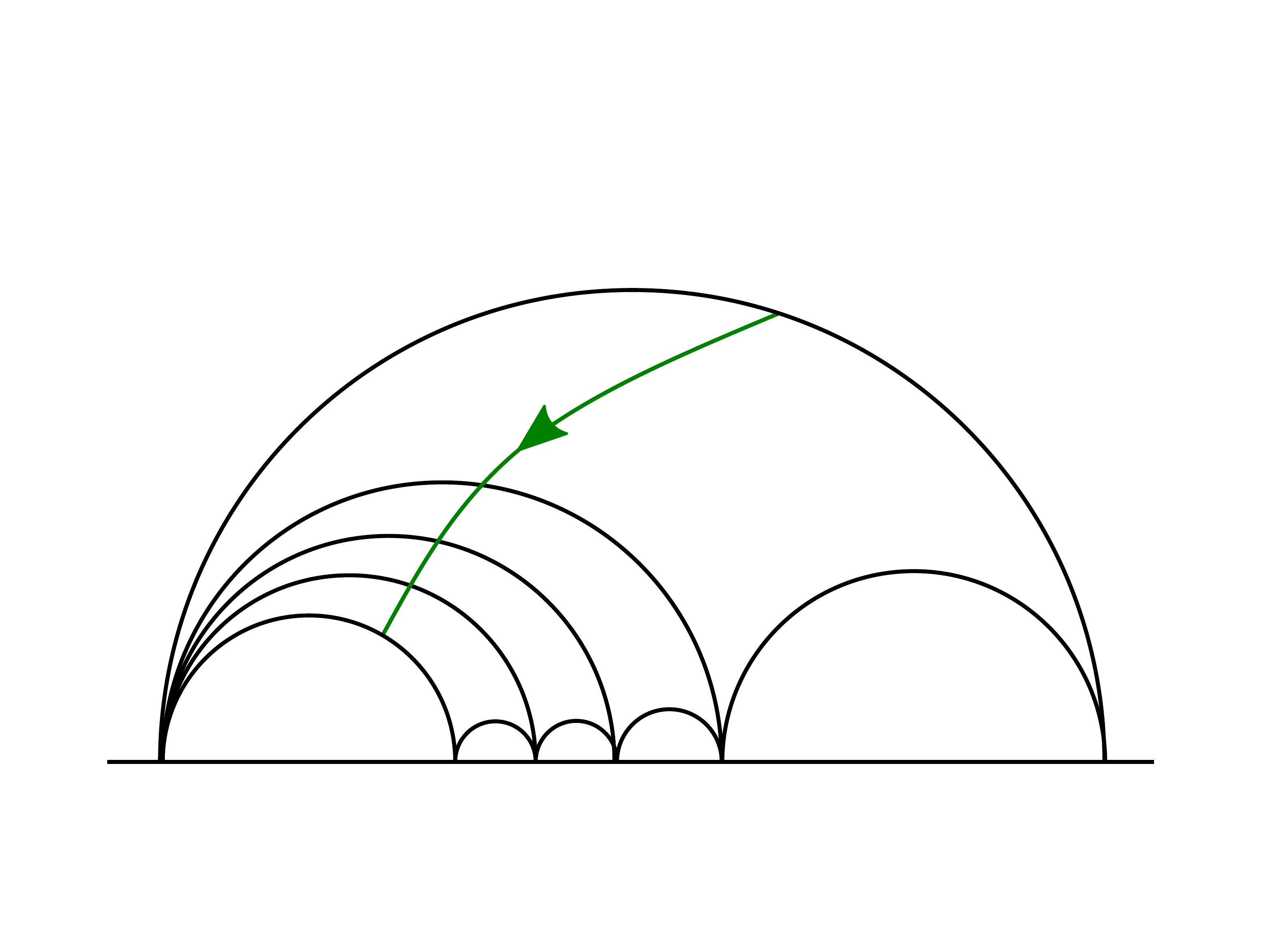}
            \caption{{An example of a right fan.}}    
       
        \end{subfigure}
        \caption{Examples of left and right triangles and fans.} 
        \label{fans}

    \end{figure}

Let $T$ be an ideal triangulation of $\mathbb{H}$ and let $\zeta$ be an oriented geodesic ray, starting at some edge $E$ of $T$ and terminating at some point $p\in\partial\mathbb{H}$ (where $p$ is not an endpoint of $E$).  We can then form an ordered collection $\{\tau_i\}_{i\in\mathbb{N}\cup\{0\}}$  of the all the triangles in $T$, which $\zeta$ non-trivially intersects, \textit{i.e.} $\zeta$ intersects the interior of each triangle $\tau_i$. For each triangle $\tau_i$, the geodesic ray $\zeta$ either cuts $\tau_i$ to form a left triangle, a right triangle, or terminates at the opposing vertex. If $\zeta$ intersects multiple left triangles in a row, then we refer to the collection of all these triangles as a \textit{left fan}. Similarly, if $\zeta$ intersects multiple right triangles in a row, then we refer to the collection of right triangles as a \textit{right fan}. See Fig.~\ref{fans} (c) and (d). If $\zeta$ passes through an opposing vertex of a triangle $\tau$, then we could think of this as $\zeta$ cutting $\tau$ to form either a left triangle or a right triangle - however, for the sake of uniqueness, we will always take this triangle to be a left triangle. If $\zeta$ terminates at an opposing vertex, then $\zeta$ does not intersect any more triangles in $T$. In particular, the collection of triangles that $\zeta$ intersects is finite if and only if $\zeta$ terminates at some opposing vertex.

Using these notions, we can define the \textit{cutting sequence} $(\zeta,T)$ of a geodesic ray $\zeta$ relative to a triangulation $T$, as follows:

\begin{defn}
Let $T$ be an ideal triangulation of $\mathbb{H}$, let $E$ be any edge of $T$ and let $\zeta$ be an oriented geodesic ray starting at $E$ and terminating at some point $p\in\partial{\mathbb{H}}$.  Also, let $\{\tau_i\}_{i\in\mathbb{N}\cup\{0\}}$ be the ordered collection of all triangles in $T$ which $\zeta$ non-trivially  intersects. Then, the $\textit{cutting sequence}$ of $\zeta$ with respect to $T$, denoted $(\zeta,T)$, is the (potentially) infinite word over the alphabet $\{L,R\}$, formed by the following algorithm:
\begin{enumerate}
\item Start with $i=0$ and $(\zeta,T)=\varepsilon$.
\item Repeat the following process until told to stop:
\begin{itemize}
\item If $\zeta$ cuts $\tau_i$ to form a  left triangle:
\begin{itemize}
\item Append the letter $L$ to $(\zeta,T)$.
\item Set $i=i+1$.
\end{itemize}
\item Else, if $\zeta$ cuts $\tau_i$ to form a  right triangle:
\begin{itemize}
\item Append the letter $R$ to $(\zeta,T)$.
\item Set $i=i+1$.
\end{itemize}
\item Else, $\zeta$ intersects the opposing vertex of $\tau_i$:
\begin{itemize}
\item Append $L$ to $(\zeta,T)$.
\item Stop.
\end{itemize}
\end{itemize}
\item End of algorithm.
\end{enumerate}
\end{defn}

We can write every cutting sequence $(\zeta,T)$ in the form $L^{n_0}R^{n_1}L^{n_2}\cdots$, where $n_0\in\mathbb{N}\cup\{0\}$ and $n_i\in\mathbb{N}$. Each index $n_i$ indicates the size of the $i$-th fan which $\zeta$ forms with $T$. We will abuse notation and also refer to the term $L^{n_i}$/$R^{n_i}$ in the cutting sequence as the \textit{$i$-th fan} of the cutting sequence  $(\zeta,T)$.

Since we can write each cutting sequence in the form $L^{n_0}R^{n_1}L^{n_2}\cdots$ for $n_0\in\mathbb{N}\cup\{0\}$ and $n_i\in\mathbb{N}$, there is an natural map $\eta$ between cutting sequences and continued fraction expansions of positive real numbers. This map converts each fan of size of $n_i$ into a partial quotient of size $n_i$. Explicitly, we have $\eta:L^{n_0}R^{n_1}L^{n_2}\cdots\mapsto[n_0;n_1,n_2,\ldots]$. If the cutting sequence is finite, then it maps to a finite continued fraction. If the cutting sequence is infinite, then it maps to an infinite continued fraction.

\begin{rem}If we have the cutting sequence $L^{n_0}R^{n_1}L^{n_2}\cdots{L^{n_k}L}$, then this would correspond to the continued fraction $[n_0;n_1,n_2,\ldots,n_k+1]$. In our convention, we will always take $L$ to be the final term. This ensures that the cutting sequence is formed in a unique way. However, we could have instead picked $R$ to be our final term, \textit{i.e.} $L^{n_0}R^{n_1}L^{n_2}\cdots{L^{n_k}R}$. This would correspond to the continued fraction $[n_0;n_1,n_2,\ldots,n_k,1]$. In particular, the choice of ending the cutting sequence with either  $L$ or $R$ is analogous to the choice of whether the continued fraction expansion is of the form $[n_0;n_1,n_2,\ldots,n_k+1]$ or $[n_0;n_1,n_2,\ldots,n_k,1]$.
\end{rem}

\subsubsection{The Farey Tessellation $\mathcal{F}$}\label{FT}
The Farey tessellation $\mathcal{F}$ is an ideal triangulation of the upper-half plane $\mathbb{H}$. The vertices are the set $\mathbb{Q}\cup\{\infty\}$. Two vertices $A$ and $B$ have a geodesic edge between them if once written in reduced form, $A=\frac{p}{q}$ and $B=\frac{r}{s}$, we have $\mid{ps-qr}\mid=1$. We will say that two vertices are \textit{neighbours}, if they have an edge between them. In this definition, we treat $\infty$ as $\frac{1}{0}$.

Given two vertices $A=\frac{p}{r}$ and $B=\frac{q}{s}$ in $\mathbb{Q}\cup\{\infty\}$, written in reduced form, we can define \textit{Farey addition} $\oplus$ and \textit{Farey subtraction} $\ominus$, as follows:
\[A\oplus{B} := \frac{p+r}{q+s}=\frac{r+p}{s+q}=:B\oplus{A} \]
\[A\ominus{B} := \frac{p-r}{q-s}=\frac{r-p}{s-q}=:B\ominus{A}\]

The first thing to note is that if $A=\frac{p}{q}$ and $B=\frac{r}{s}$  are neighbours in the Farey tessellation, \textit{i.e.} $\mid{ps-qr}\mid=1$, then the point $A\oplus{B}=\frac{p+r}{q+s}$ is a neighbour of both $A$ and $B$. The points $A$ and $A\oplus{B}$ are neighbours since:
$$ \mid{}p\cdot(q+s)-q\cdot(p+r)\mid{}=\mid{pq+ps-qp-qr\mid{}}=\mid{ps-qr}\mid=1,$$
and the points $B$ and $A\oplus{B}$ are neighbours since:
$$ \mid{}r\cdot(q+s)-s\cdot(p+r)\mid{}=\mid{rq+rs-ps-sr\mid{}}=\mid{-ps+qr}\mid=1.$$

As a result, the points $A,B$ and $A\oplus{B}$ each have a geodesic edge between them, and, therefore, form a triangle in $\mathcal{F}$. Similarly, if $A$ and $B$ are neighbours in the Farey tessellation, then the point $A\ominus{B}$ is also a neighbour of both $A$ and $B$ (and is not a neighbour of $A\oplus{B}$). 

If we start with the points $\frac{0}{1}$ and $\frac{1}{0}$, then we can generate all points in $\mathbb{Q}\cup\{\infty\}$ by using iterative Farey addition and Farey subtraction. See \cite{Series2:1985}. See Fig.~\ref{Conv} for a truncated picture of the Farey tessellation (in Section~\ref{CSFT}).

Given any point $z\in\mathbb{H}$ and any matrix $M=\begin{psmallmatrix} a & b \\ c& d\end{psmallmatrix}\in{PSL_2(\mathbb{R})}$, we can define the action of $M$ on each point $z\in\mathbb{H}$ as follows: $$M\cdot{z}:=\frac{az+b}{cz+d}.$$ The group $PSL_2(\mathbb{R})$ with action as defined above is isomorphic to the group of \textit{orientation preserving isometries} of $\mathbb{H}$, denoted $Isom^+(\mathbb{H})$. 

If we take $M=\begin{psmallmatrix} p & r \\ q & s\end{psmallmatrix}\in{PSL_2(\mathbb{Z})}<PSL_2(\mathbb{R})$, and we take the line $I$ between $0$ and $\infty$, then  the action of $M$ on $I$ maps $I$ to an edge between the points $M\cdot{0}=\frac{r}{s}$ and $M\cdot{\infty}=\frac{p}{q}$. Since $M\in{PSL_2(\mathbb{Z})}$, it follows that $\det(M)=ps-qr=1$. As a result, $M$ maps $I$ to an edge of $\mathcal{F}$. Alternatively, if $A=\frac{p}{q}$ and $B=\frac{r}{s}$ are neighbours in $\mathcal{F}$, then, since $\mid{ps-rq}\mid=1$, it follows trivially that either $\begin{psmallmatrix} p & r \\ q & s\end{psmallmatrix}$ or $\begin{psmallmatrix} p & -r \\ q & -s\end{psmallmatrix}$ is an element of $PSL_2(\mathbb{Z})$. This gives us the following proposition:

\begin{prop}\label{PSL}
Two points $A=\frac{p}{q}$ and $B=\frac{r}{s}$ are neighbours in $\mathcal{F}$ if and only if  either $\begin{psmallmatrix} p & r \\ q & s\end{psmallmatrix}$ or $\begin{psmallmatrix} p & -r \\ q & -s\end{psmallmatrix}$ is an element of $PSL_2(\mathbb{Z})$. 
\end{prop}

With the above proposition in mind, we can easily deduce that the set of edges of the Farey tessellation is equivalent to the set of edges $PSL_2(\mathbb{Z})\cdot{I}$, \textit{i.e.} the set of images of $I$ under the action of $PSL_2(\mathbb{Z})$. This allows us to deduce that $\mathcal{F}$ is \textit{preserved} under the action of $PSL_2(\mathbb{Z})$, \textit{i.e.} $M\cdot{\mathcal{F}}=\mathcal{F}$ for all $M\in{PSL_2(\mathbb{Z})}$. Furthermore, $PSL_2(\mathbb{Z})$ is the \textit{maximal orientation-preserving group} which preserves $\mathcal{F}$, \textit{i.e.} $M\cdot{\mathcal{F}}\neq\mathcal{F}$ for any $M\in{PSL_2(\mathbb{R})\setminus{PSL_2(\mathbb{Z})}}$. We write $Isom^+({\mathcal{F}})=PSL_2(\mathbb{Z})$ to indicate that $PSL_2(\mathbb{Z})$ is the {maximal orientation-preserving group} which preserves $\mathcal{F}$.

\subsubsection{Cutting Sequences and the Farey Tessellation}\label{CSFT}
   
The following theorem highlights the importance of the Farey tessellation with regards to continued fractions. Recall that $\eta$ is the map the converts cutting sequences into continued fractions expansions, \textit{i.e.} $\eta:L^{n_0}R^{n_1}\cdots\mapsto[n_0;n_1,\ldots]$.

\begin{thm}\label{thma}\emph{(\cite[Theorem A]{Series2:1985})} Let $\zeta$ be a geodesic in $\mathbb{H}$ with  endpoints $\alpha_1>0$ and $\alpha_2<0$, and let $I$ be the geodesic line between $0$ and $\infty$. Let $I_+$ be the region $\{z:Re(z)>0\}$ and $I_-$ be the region $\{z:Re(z)<0\}$. Then, for $\zeta^+=\zeta\cap{{I_+}}$ and $\zeta^-=\zeta\cap{{I_-}}$ (with implicit orientation), $\eta((\zeta^+,\mathcal{F}))$ is the continued fraction expansion of $\alpha_1$ and $\eta((\zeta^-,\mathcal{F}))$ is the continued fraction expansion of $\frac{-1}{\alpha_2}$.
\end{thm}

The main point we take away from the above theorem is the following: if $\zeta_\alpha$ is a geodesic ray starting at the the $y$-axis $I$ and terminating at the the point $\alpha\in\mathbb{R}_{>0}$, then $\eta((\zeta_\alpha,\mathcal{F}))=\overline{\alpha}$. As a result, we can identify the real number $\alpha\in\mathbb{R}_{>0}$ with any geodesic ray $\zeta_\alpha$ starting at $I$ and terminating at the point $\alpha$, and the cutting sequence $(\zeta_\alpha,\mathcal{F})$ is equivalent to the continued fraction expansion $\overline\alpha$. However, this is not the only connection between the cutting sequence of a geodesic ray $\zeta_\alpha$ with the Farey tessellation and the continued fraction expansion $\overline{\alpha}$.

\begin{figure}[htb]
\centering
  \includegraphics[width=0.8\linewidth]{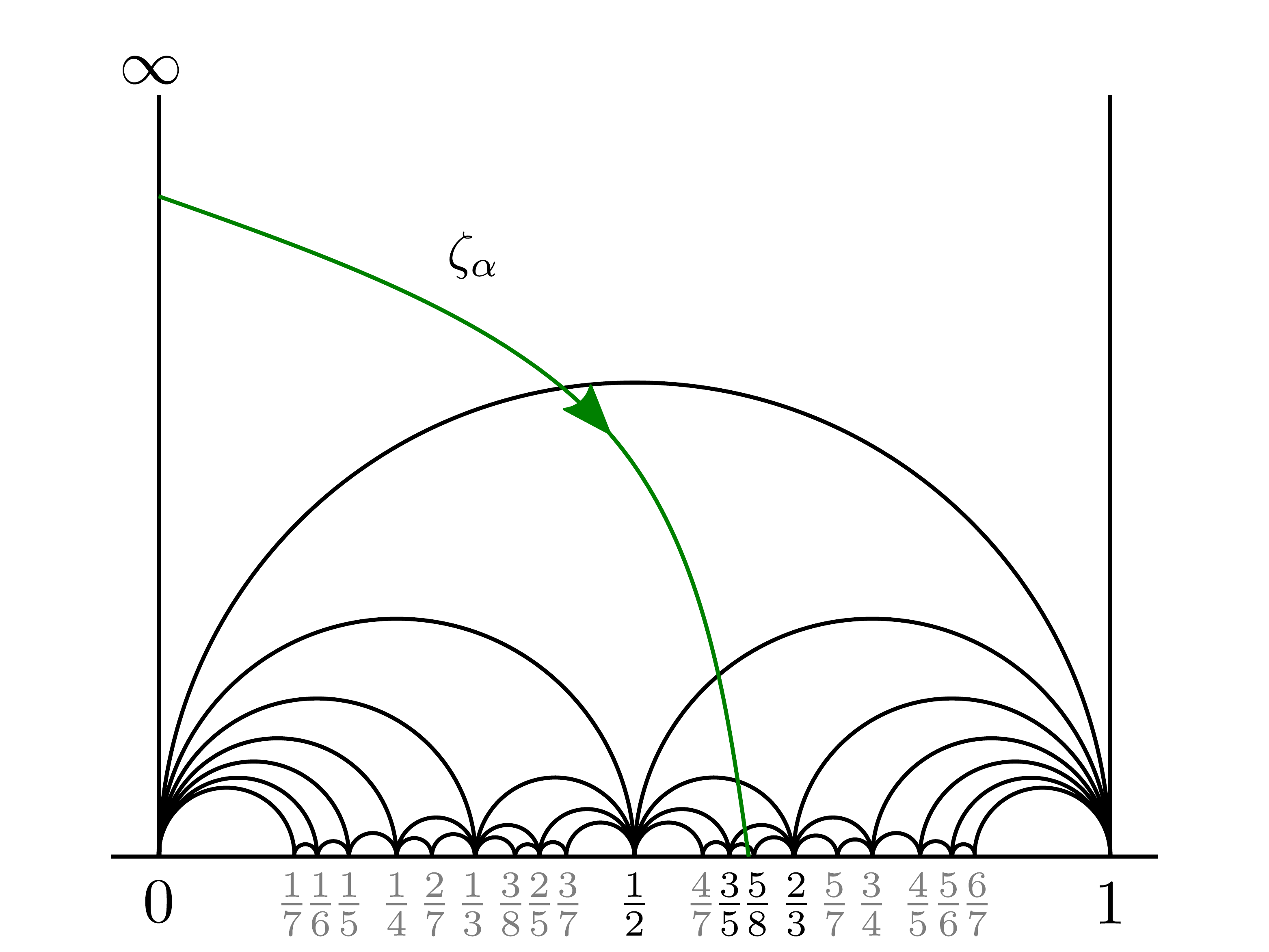} 
  \caption{An image of a geodesic ray $\zeta_{\alpha}$ intersecting the Farey tessellation $\mathcal{F}$ with (some of the) convergents shown in bold. The endpoint of $\zeta_\alpha$ is $\alpha=\frac{\sqrt{5}-1}{2}$. The convergents are $\infty,0,1,\frac{1}{2},\frac{2}{3},\frac{3}{5},\frac{5}{8},\ldots$.}
  \label{Conv}
\end{figure}

\begin{cor}\label{semiconvedge}
Let $\alpha\in\mathbb{R}_{>0}$.
If $A=\frac{p_k}{q_k}$ is the $k$-th convergent of $\alpha$ and $B=\frac{p_{\{k,m\}}}{q_{\{k,m\}}}:=\frac{mp_k+p_{k-1}}{mq_k+q_{k_1}}$ is a $\{k,m\}$-th semi-convergent, then there is an edge $E$ between $A$ and $B$ in $\mathcal{F}$. Moreover, the corresponding geodesic ray $\zeta_\alpha$ intersects $E$.
\end{cor}

\begin{proof}
The first thing to note is that if $$\frac{p_{\{k,m\}}}{q_{\{k,m\}}}:=\frac{mp_k+p_{k-1}}{mq_k+q_{k_1}},$$ for some $m\in\mathbb{N}$, then by extension of a classical theorem of continued fractions, we have:
$$p_{\{k,m\}}q_k-p_kq_{\{k,m\}}=(-1)^{k}.$$
See \cite[Theorem~150]{HW:1938}.

In particular, $\frac{p_{\{k,m\}}}{q_{\{k,m\}}}$ and $\frac{p_k}{q_k}$ are neighbours in $\mathcal{F}$ and connected by some edge $E$. If $E=I$, then since is defined to $\zeta_\alpha$ start at $I$, it follows trivially that $\zeta_\alpha$ intersects $I$. We will assume that $E\neq{I}$. Similarly, if $\alpha=\frac{p_{\{k,m\}}}{q_{\{k,m\}}}$ or $\alpha=\frac{p_k}{q_k}$, then the corresponding geodesic ray trivially intersects $E$ (\textit{i.e.} at an endpoint).

Otherwise by extension of \cite[Theorem~167]{HW:1938}, the point $\alpha$ lies between the points $\frac{p_{\{k,m\}}}{q_{\{k,m\}}}$ and $\frac{p_k}{q_k}$, \textit{i.e.}: $$\alpha\in\left(\min\left\{\frac{p_{\{k,m\}}}{q_{\{k,m\}}},\frac{p_k}{q_k}\right\},\max\left\{\frac{p_{\{k,m\}}}{q_{\{k,m\}}},\frac{p_k}{q_k}\right\}\right).$$ The edge $E$ then separates $\mathbb{H}$ into two regions: one containing $I$ and the other containing the interval: $$\left(\min\left\{\frac{p_{\{k,m\}}}{q_{\{k,m\}}},\frac{p_k}{q_k}\right\},\max\left\{\frac{p_{\{k,m\}}}{q_{\{k,m\}}},\frac{p_k}{q_k}\right\}\right).$$ Since $\zeta_\alpha$ starts at $I$ and terminates at $\alpha$, which are contained in these two distinct regions, we can conclude that $\zeta_\alpha$ must intersect $E$ to pass from one region to the other, as required.
\end{proof}

\begin{cor}\label{edgesemiconv} Let $\zeta_\alpha$ be a geodesic ray that starts at $I$ and terminates at $\alpha\in\mathbb{R}_{>0}$. If $\zeta_\alpha$ intersects an edge $E$ in $\mathcal{F}$, then at least one of the vertices of this edge will be a convergent and the other will be a semi-convergent (and possibly even a convergent).
\end{cor}

\begin{proof}
For every fan that $\zeta_\alpha$ forms with $\mathcal{F}$, there is a vertex which is in all of the triangles of this fan. In particular, every edge in this fan will have a unique common endpoint. We refer to this vertex as the \textit{fixed vertex} of the fan. Let $v_k$ be the fixed vertex of the $(k+1)$-th fan. Then we can label each edge in the fan $E_{k,i}$, where $0\leq{i}\leq{n_{k+1}}$, using the order that $\zeta_\alpha$ intersects these edges. As previously mentioned, each of these edges $E_{k,i}$ has a common vertex $v_k$. For each edge $E_{k,i}$, we label the other vertex $v_{k,i}$. If $v_{k,i}$ is the final ``other vertex'' in this fan,  (\textit{i.e.} $i=n_{k+1}$), then this vertex is either the endpoint of $\zeta_\alpha$ or it is the fixed vertex of the next fan, \textit{i.e.} the $(k+2)$-th fan. Likewise, if $v_{k,0}$ is the first ``other vertex'' of the $(k+1)$-th fan, then $v_{k,0}$ is the fixed vertex of the previous fan, \textit{i.e.} the $k$-th fan. We can now note, that if $(\zeta_\alpha,\mathcal{F})=L^{n_0}R^{n_1}\cdots$, then the geodesic ray $\zeta_\alpha^{k,i}$ which starts at $I$ and terminates at the vertex $v_{k,i}$, has cutting sequence:  $$(\zeta_\alpha^{k,i},\mathcal{F})=L^{n_0}R^{n_1}\cdots{L^{n_k}}R^{i-1}L$$ or $$(\zeta_\alpha^{k,i},\mathcal{F})=L^{n_0}R^{n_1}\cdots{R^{n_k}}L^i, $$depending on whether $k$ is even or odd respectively. As a result, we find that the point $v_{k,i}$ has continued fraction expansion $[n_0;n_1,\ldots,n_k,i]$ (up to taking equivalent continued fraction expansions). However, this is simply the $\{k,i\}$-th semi-convergent of $\alpha$. See Definition~\ref{semiconv}. Note that by construction, the point $v_k=v_{k-1,n_k}=v_{k+1,0}$ is the $k$-th convergent $\frac{p_k}{q_k}$, which can also be written as the $\{k-1,n_k\}$-th semi-convergent $\frac{p_{k-1,n_k}}{q_{k-1,n_k}}$.

Except for possibly the point $\frac{p_{-1}}{q_{-1}}=\frac{1}{0}=\infty$, every convergent is a fixed point of a fan. This means that each convergent is the endpoint of at least two edges that $\zeta_\alpha$ intersects. Alternatively, if $\zeta_\alpha$ intersects two distinct edges, which have the same endpoint, then this endpoint is a fixed point of a fan and, therefore, this point is a convergent.
\end{proof}

\section{\texorpdfstring{Infinite Loops and the $p$-adic Littlewood Conjecture}{Infinite Loops and the p-adic Littlewood Conjecture}}\label{Infloop}

In this section, we will discuss \textit{infinite loops} mod $n$. In Section~\ref{mult}, we will motivate the concept of an infinite loop by first looking at how replacing the Farey tessellation $\mathcal{F}$ with the \textit{$\frac{1}{n}$-scaled Farey tessellation} $\frac{1}{n}\mathcal{F}$ induces integer multiplication by $n$ with respect to the corresponding cutting sequences, \textit{i.e.} if the cutting sequence $(\zeta_\alpha,\mathcal{F})$ corresponds to the continued fraction expansion of $\alpha$, then $(\zeta_\alpha,\frac{1}{n}\mathcal{F})$ corresponds to the continued fraction expansion of $n\alpha$. See~\cite{paper}.

Since the $p$-adic Littlewood Conjecture is closely related to the behaviour of continued fractions under integer multiplication, looking at the structure of $\mathcal{F}$ and $\frac{1}{n}\mathcal{F}$, as well as how these structures interact, seems like a natural place to find out more information about pLC. What we find is that if a geodesic ray $\zeta_\alpha$ intersects some edge $E$ in $\mathcal{F}\cap\frac{1}{n}\mathcal{F}$, then we can deduce that at least some of the convergents of $\alpha$ influence the convergents of  $n\alpha$ in a very nice, direct way. However, if $\zeta_\alpha$ does not intersect $\mathcal{F}\cap\frac{1}{n}\mathcal{F}$, then no convergents of $\alpha$ influence any convergents of $n\alpha$ directly. In particular, the continued fraction expansions of such geodesics behave ``badly'' relative to integer multiplication.

Instead of looking at geodesic rays which do not intersect $\mathcal{F}\cap\frac{1}{n}\mathcal{F}$, we will look at geodesic rays which satisfy a slightly weaker property: geodesic rays $\zeta_\alpha$ which do not intersect $\Gamma_0(n)\cdot{I}\subset\mathcal{F}\cap\frac{1}{n}\mathcal{F}$. This leads to the definition of an infinite loop mod $n$:

\begin{defnum}[\ref{infloopdef} (a)]
Let $\zeta_\alpha$ be a geodesic ray starting at the $y$-axis $I$ and terminating at the point $\alpha\in\mathbb{R}_{>0}$. Then $\zeta_\alpha$ is an \textit{infinite loop mod} $n$, if $\zeta_\alpha$ is disjoint from $\Gamma_0(n)\cdot{I}$  except for the edges of the form $I+k$, for $k\in\mathbb{Z}_{\geq{0}}$.
\end{defnum}

As we will see in Section~\ref{ILpLC2}, if $\alpha$ is not an infinite loop mod $n$, then we can get some nice information about $B(\alpha)$ and $B(n\alpha)$. In particular, $B(\alpha)$ and $B(n\alpha)$ can not both be small relative to $\sqrt{n}$. Furthermore, if there is some fixed $m\in\mathbb{N}$ such that $p^\ell\alpha$ is an infinite loop mod $p^m$ for all $\ell\in\mathbb{N}$, then $\alpha$ is a counterexample to pLC. These two facts combine together to give us the following reformulation of pLC:

\begin{thmn}[\ref{theorem}]
Let $\alpha\in\textbf{Bad}$. Then
$\alpha$ satisfies pLC if and only if there is a sequence of natural numbers $\left\{\ell_m\right\}_{m\in\mathbb{N}}$ such that $p^{\ell_m}\alpha$ is not an infinite loop mod $p^m$.
\end{thmn}

\subsection{Multiplication Described by Triangulation Replacement of Cutting Sequences}\label{mult}

Let $n^*:=\begin{psmallmatrix} \sqrt{n} & 0\\ 0 & \frac{1}{\sqrt{n}}\end{psmallmatrix} \in{PSL_2(\mathbb{R})}$ and define $\frac{1}{n^*}:=(n^*)^{-1}$ for $n\in\mathbb{N}$. These two maps scale both $\mathbb{H}$ and $\mathcal{F}$ by a factor of $n$ and $\frac{1}{n}$, respectively. In particular, they multiply the real axis by $n$ and $\frac{1}{n}$, respectively. For example, if $x\in\mathbb{R}$, then $n^*\cdot{x}=nx$ and $(n^*)^{-1}\cdot{}x=\frac{x}{n}$. Since $n^*\not\in{PSL_2(\mathbb{Z})}$ for $n>{1}$, these maps do not preserve $\mathcal{F}$ and we will refer to the images of $\mathcal{F}$ under these maps as $n\mathcal{F}$ and $\frac{1}{n}\mathcal{F}$, respectively. Both $n\mathcal{F}$ and $\frac{1}{n}\mathcal{F}$ will be ideal triangulations of $\mathbb{H}$, since the $n^*$ map will take geodesics to geodesics and triangles to triangles. It is worth noting that both of these maps preserve the line $I$ between $0$ and $\infty$, which is our conventional starting edge for our geodesic rays in $\mathcal{F}$.  It follows that for any geodesic ray $\zeta_\alpha$ starting at $I$ and terminating at $\alpha\in\mathbb{R}_{>0}$, the scaled geodesic ray $n^*(\zeta_\alpha)$ will also start at $I$ and terminate at the point $n\alpha\in\mathbb{R}_{>0}$. Note that $n^*(\zeta_\alpha)$ will also be a geodesic ray, since $n^*\in{PSL_2(\mathbb{R})}\cong{Isom^+(\mathbb{H})}$. As a result, the cutting sequence $(n^*(\zeta_\alpha),\mathcal{F})$ will be equivalent to the continued fraction expansion of $n\alpha$. 

Alternatively, we can scale the Farey tessellation by $(n^*)^{-1}$ to get the tessellation $\frac{1}{n}\mathcal{F}$. Relatively speaking, the geodesic ray $n^*(\zeta_\alpha)$ will intersect $\mathcal{F}$ in the same way that $\zeta_\alpha$ intersects $\frac{1}{n}\mathcal{F}$. Therefore, the cutting sequences will be equivalent, \textit{i.e.}  $(\zeta,\frac{1}{n}\mathcal{F})=(n^*(\zeta),\mathcal{F})$, and so, $\eta(\zeta_\alpha,\frac{1}{n}\mathcal{F})=\eta(n^*(\zeta_\alpha),\mathcal{F})=\overline{n\alpha}$. As a result, we can view the integer multiplication map of continued fractions $\overline{n}:\overline{\alpha}\rightarrow{\overline{n\alpha}}$ as being equivalent to replacing the triangulation $\mathcal{F}$ with $\frac{1}{n}\mathcal{F}$ in the corresponding cutting sequence. Explicitly, we can express $\overline{n}$ as the map between the cutting sequences $\overline{n}:\eta(\zeta_\alpha,\mathcal{F})\rightarrow{\eta(\zeta_\alpha,\frac{1}{n}\mathcal{F})}$. 

As a consequence, if we want to understand multiplication of continued fractions, it will be useful to further investigate the structure of $\frac{1}{n}\mathcal{F}$ relative to $\mathcal{F}$.

\subsubsection{The Structure of $\mathcal{F}\cap\frac{1}{n}\mathcal{F}$}

Recall from Section~\ref{FT}, that two points $A=\frac{p}{q}$ and $B=\frac{r}{s}$ in $\mathbb{Q}\cup\left\{\infty\right\}$ are neighbours in $\mathcal{F}$ if and only if  $|ps-rq|=1$. This in turn implies that there is some element $M\in{PSL_2(\mathbb{Z})}$ such that $M\cdot\infty=A$ and $M\cdot{0}=B$. This matrix $M$ is either of the form $\begin{psmallmatrix} p & r \\ q& s \end{psmallmatrix}$ or $\begin{psmallmatrix} p & -r \\ q& -s \end{psmallmatrix}$, depending on whether $ps-rq=1$ or $ps-rq=-1$, respectively. It is important to note that $A$ and $B$ can only be neighbours in $\mathcal{F}$ if $\gcd(ps,rq)=1$. By extension we must have that $\gcd(p,r)=\gcd(q,s)=1$. 

Using this information about $\mathcal{F}$, we can deduce similar information about $\frac{1}{n}\mathcal{F}$ by simply scaling $\mathcal{F}$ by the $(n^*)^{-1}$ map. Using this structure, we obtain the following lemma:

\begin{lem}\label{lem2} Two points $A$ and $B$ are neighbours in both $\mathcal{F}$ and $\frac{1}{n}\mathcal{F}$ if and only if they have reduced form $\frac{a}{n_{1}c_1}$ and $\frac{b}{n_2d_1}$, with $n=n_1 n_2$ and $|{an_2d_1-bn_1c_1}|=1$.
\end{lem}

\begin{proof}
$(\Rightarrow)$: Assume that $A=\frac{a}{c}$ and $B=\frac{b}{d}$ are neighbours in $\mathcal{F}$ and $\frac{1}{n}\mathcal{F}$. Since $A$ and $B$ are neighbours in $\mathcal{F}$, we can conclude that $|ad-bc|=1$, and more importantly for us: $$\gcd(c,d)=1.$$ 
Since $\frac{1}{n}\mathcal{F}$ is a scaled version of the Farey tessellation, $A$ and $B$ are neighbours in $\frac{1}{n}\mathcal{F}$ if and only if $n^*\cdot{A}=n\cdot{A}=\frac{na}{c}$ and $n^*\cdot{B}=n\cdot{B}=\frac{nb}{d}$ are neighbours in $\mathcal{F}$. Of course, $n\cdot{A}=\frac{na}{c}$ and $n\cdot{B}=\frac{nb}{d}$ will not necessarily be in reduced form. We will take $g:=\gcd(c,n)$ and $h:=\gcd(d,n)$. In this case, we can rewrite $c,d$ and $n$ in the following ways:
\begin{align*}
c=n_1c_1, \quad & n=n_1g, \\d=n_2d_1,  \quad & n=n_2h.
 \end{align*}
We can then rewrite $n\cdot{A}$ and $n\cdot{B}$ in reduced form as:
 $$n\cdot{}A=\frac{n_1ga}{n_1c_1}=\frac{ga}{c_1},$$ $${n}\cdot{}B=\frac{n_2hb}{n_2d_1}=\frac{hb}{d_1}.$$ 
Since $n\cdot{A}$ and $n\cdot{B}$ are neighbours in $\mathcal{F}$, we see that $|gad_1-hbc_1|=1$. Necessarily, we can not have $\gcd(g,h)=r\neq{1}$, since this would imply that $|gad_1-hbc_1|\equiv{0}\mod{r}$ and so $|gad_1-hbc_1|\neq{}1$. Therefore, we can conclude that: $$\gcd(g,h)=1.$$ 
Since we know that $gcd(c,d)=1$,  $c=n_1c_1$, and $d=n_2d_1$, we can conclude that:
 $$\gcd(c,d)=1=\gcd(n_1c_1,n_2d_1)=\gcd(n_1,n_2).$$
 Using this equality, we see that:
\begin{align*} n_1 &=\gcd(n_1,n)\\ &=\gcd(n_1,n_2h)\\ &=\gcd(n_1,n_2)\cdot\gcd(n_1,h)\\ &=1\cdot\gcd(n_1,h)\\ &=\gcd(n_1,h).
 \end{align*}
However, since $\gcd(g,h)=1$, we can also deduce that:  \begin{align*}
h &=\gcd(h,n)\\ &=\gcd(h,n_1g)\\ &=\gcd(h,n_1)\cdot\gcd(h,g)\\&=\gcd(h,n_1)\cdot1\\ &=\gcd(h,n_1),
\end{align*}
 and so: $$n_1=\gcd(n_1,h)=\gcd(h,n_1)=h.$$
 Since $n=n_1g=n_2h$, we can now conclude that $g=n_2$, and so: $$n=n_1n_2.$$
Combining this information all together, we can now write $A=\frac{a}{n_1c_1}$ and $B=\frac{b}{n_2d_1}$ with $|an_2d_1-bn_1c_1|=1$ and $n=n_1n_2$, as required.

$(\Leftarrow):$ Let $A=\frac{a}{n_{1}c}$ and $B=\frac{b}{n_2d}$ with $n=n_1 n_2$ and $|{an_2d-bn_1c}|=1$. Since $|{an_2d-bn_1c}|=1$, we see that $A$ and $B$ are neighbours in $\mathcal{F}$. Writing $n\cdot{A}$ and $n\cdot{B}$ in reduced form, we have that: $$n\cdot{A}=\frac{n_2a}{c},$$ and $$n\cdot{B}=\frac{n_1b}{d}.$$
We can now check to see if $n\cdot{A}$ and $n\cdot{B}$ are neighbours in $\mathcal{F}$ by computing the value of  $|{n_2ad-n_1bc}|$. Here, we have  $|{n_2ad-n_1bc}|=|{an_2d-bn_1c}|=1$, and so $n\cdot{A}$ and $n\cdot{B}$ are indeed neighbours in $\mathcal{F}$. By rescaling by a factor of $(n^*)^{-1}$, we now see that $A$ and $B$ are neighbours in $\frac{1}{n}\mathcal{F}$, as required.
\end{proof}

In the above lemma (Lemma~\ref{lem2}), requiring the condition that $A$ and $B$ have reduced form $\frac{a}{n_{1}c}$ and $\frac{b}{n_2d}$ with $n=n_1 n_2$ and $|{an_2d-bn_1c}|=1$, is equivalent to saying that if $A$ and $B$ are neighbours of this form in either $\mathcal{F}$ or $\frac{1}{n}\mathcal{F}$, then necessarily they are neighbours in both $\mathcal{F}$ and $\frac{1}{n}\mathcal{F}$. 

\subsubsection{Geodesics Intersecting $\mathcal{F}\cap\frac{1}{n}\mathcal{F}$}\label{FnF}

Assume that $\zeta_\alpha$ is a geodesic ray which starts at $I$ and terminates at $\alpha$. If $\zeta_\alpha$ intersects an edge $E$ in $\mathcal{F}\cap\frac{1}{n}\mathcal{F}$, then we can cut $\zeta_\alpha$ up into two pieces:
$\zeta_{\alpha,1}$, which runs along $\zeta_\alpha$ from $I$ to $E$, and $\zeta_{\alpha,2}$, which runs along $\zeta_\alpha$ from $E$ to $\alpha$. Note that $\zeta_{\alpha,1}$ starts at an edge $I$ in $\mathcal{F}$ and terminates at an edge $E$ in $\mathcal{F}$. Therefore, if $\zeta_{\alpha,1}$ intersects a triangle in $\mathcal{F}$, then it cuts this triangle to either form a left triangle or a right triangle. As a result, we can produce a well-defined cutting sequence $(\zeta_{\alpha,1},\mathcal{F})$ - even though $\zeta_{\alpha,1}$ is a geodesic segment and not a geodesic ray. Similarly, since $\zeta_{\alpha,2}$ starts at an edge $E$ in $\mathcal{F}$, the cutting sequence $(\zeta_{\alpha,2},\mathcal{F})$ is also well-defined. Furthermore, when we cut along $E$ to produce $\zeta_{\alpha,1}$ and $\zeta_{\alpha,2}$, we effectively split the cutting sequence $(\zeta_\alpha,\mathcal{F})$ into two smaller cutting sequences. These are exactly the cutting sequences $(\zeta_{\alpha,1},\mathcal{F})$ and $(\zeta_{\alpha,2},\mathcal{F})$. Gluing these cutting sequences back together we see that:
$$(\zeta_\alpha,\mathcal{F})=(\zeta_{\alpha,1},\mathcal{F})\cdot(\zeta_{\alpha,2},\mathcal{F}).$$

Similarly, since the edges $I$ and $E$ are also in $\frac{1}{n}\mathcal{F}$, we can conclude that the cutting sequences $(\zeta_{\alpha,1},\frac{1}{n}\mathcal{F})$ and $(\zeta_{\alpha,2},\frac{1}{n}\mathcal{F})$ are also well-defined. Again, we see that:
$$(\zeta_\alpha,\frac{1}{n}\mathcal{F})=(\zeta_{\alpha,1},\frac{1}{n}\mathcal{F})\cdot(\zeta_{\alpha,2},\frac{1}{n}\mathcal{F}).$$

Of course, we could do this procedure for all the edges in $\mathcal{F}\cap\frac{1}{n}\mathcal{F}$ that $\zeta_\alpha$ intersects. For example, if $\zeta_\alpha$ intersects a sequence of edges $\{E_0=I,E_1,\ldots,E_k\}$ in $\mathcal{F}\cap\frac{1}{n}\mathcal{F}$ (labelled such that $\zeta_\alpha$ cuts these edges in order), then we can cut $\zeta_\alpha$ into a sequence of geodesic segments $\{\zeta_{\alpha,1},\zeta_{\alpha,2},\ldots,\zeta_{\alpha,k+1}\}$ such that each segment $\zeta_{\alpha,i}$ runs along $\zeta_\alpha$ between $E_{i-1}$ and $E_{i}$ for $i\in\{1,\ldots,k\}$ and $\zeta_{\alpha,k+1}$ runs along $\zeta_\alpha$ from $E_k$ to $\alpha$. Since the cutting sequences $(\zeta_{\alpha,i},\mathcal{F})$ and $(\zeta_{\alpha,i},\frac{1}{n}\mathcal{F})$ are well defined for $i\in\{1,2,\ldots,k+1\}$, we see that:
$$(\zeta_\alpha,\mathcal{F})=\prod\limits_{i=1}^{k+1}(\zeta_{\alpha,i},\mathcal{F})$$ and $$(\zeta_\alpha,\frac{1}{n}\mathcal{F})=\prod\limits_{i=1}^{k+1}(\zeta_{\alpha,i},\frac{1}{n}\mathcal{F}).$$

If such a decomposition exists, then the way that the triangulation replacement (and therefore the integer multiplication) affects the cutting sequences of each geodesic segment $\zeta_{\alpha,i}$ is independent of the way that the triangulation replacement affects the cutting sequences of any other geodesic segment $\zeta_{\alpha,j}$. Heuristically, we can think of this as saying that the some of the semi-convergents of $\alpha$ directly influence some of the semi-convergents of $n\alpha$. This is formalised in the following proposition:

\begin{prop}
Let $\alpha\in\mathbb{R}_{>0}$ and assume that $\overline{\alpha}$ has a convergent of the form $\frac{p_k}{q_k}=\frac{a}{n_1c}$ and semi-convergent of the form $\frac{p_{\{k,m\}}}{q_{\{k,m\}}}=\frac{b}{n_2d}$ such that $n=n_1n_2$ and $|an_2d-bn_2c|=1$. Then $\frac{n_2a}{c}$ and $\frac{n_1b}{d}$ are both semi-convergents of $\overline{n\alpha}$. In fact, at least one of $\frac{n_2a}{c}$ or $\frac{n_1b}{d}$ will be a convergent for $\overline{n\alpha}$.
\end{prop}

\begin{proof}
Recall from Corollary~\ref{semiconvedge}, that if $\frac{p_k}{q_k}$ is a convergent of ${\alpha}$ and $\frac{p_{\{k,m\}}}{q_{\{k,m\}}}$ is a semi-convergent, then $\frac{p_k}{q_k}$ and $\frac{p_{\{k,m\}}}{q_{\{k,m\}}}$ are neighbours in $\mathcal{F}$. Furthermore, the corresponding geodesic ray $\zeta_\alpha$ intersects the edge $E$ between $\frac{p_k}{q_k}$ and $\frac{p_{\{k,m\}}}{q_{\{k,m\}}}$. Since $E$ is also an edge of $\frac{1}{n}\mathcal{F}$, we can rescale our space using the $n^*$ map. This allows us to see that $n^*\cdot{}\zeta_\alpha$ intersects $n^*\cdot{}E$, which is an edge in $\mathcal{F}$ with end points $\frac{n_2a}{c_1}$ and $\frac{n_1b}{d_1}$. Since $n^*\in{PSL_2(\mathbb{R})}$, it follows that $n^*\cdot\zeta_\alpha$ is a geodesic ray, which starts at the $y$-axis $I$, and terminates at a point $n\alpha$.  Therefore, by using Corollary~\ref{edgesemiconv}, we can conclude that both $\frac{n_2a}{c_1}$ and $\frac{n_1b}{d_1}$ will be semi-convergents of $\overline{n\alpha}$. In fact, by the proof of Corollary~\ref{semiconvedge}, one of these edges must be a fixed point in a fan. Therefore, either  $\frac{n_2a}{c_1}$ or $\frac{n_1b}{d_1}$ must be a convergent of $\overline{n\alpha}$. 
\end{proof}

\begin{rem}
In the above proof, it is worth noting that if $A=\frac{a}{n_{1}c_1}$ was a convergent of $\alpha$, this does not necessarily mean that $n{A}=\frac{n_2a}{c_1}$ is a convergent of $n\alpha$ - we can only conclude that one of $n\cdot{A}$ or $n\cdot{B}$ is a convergent.
\end{rem}

On the other hand, if we try to cut $\zeta_\alpha$ along an edge $E$ which is not in $\mathcal{F}$ (or $\frac{1}{n}\mathcal{F}$), then neither $\zeta_{\alpha,1}$ or $\zeta_{\alpha,2}$ will have well-defined cutting sequences relative to $\mathcal{F}$ (or $\frac{1}{n}\mathcal{F}$). In particular, if a geodesic ray $\zeta_\alpha$ does not intersect any edges in $\mathcal{F}\cap\frac{1}{n}\mathcal{F}$ then there is no way to decompose $\zeta_\alpha$ into smaller geodesic segments, such that each geodesic segment has a well-defined cutting sequence relative to both $\mathcal{F}$ and $\frac{1}{n}\mathcal{F}$. Again, heuristically we can think of this as essentially saying that no semi-convergents of $\alpha$ directly correspond any semi-convergents of $n\alpha$. Put another way, if we have an algorithm that maps partial quotients of $\overline{\alpha}$ to partial quotients of $\overline{n\alpha}$ and $\zeta_\alpha$ does not intersect $\mathcal{F}\cap\frac{1}{n}\mathcal{F}$ (except for at $I$), then each step in the algorithm will always depend on previous  partial quotients of $\overline{\alpha}$. As a result, these numbers behave badly with respect to integer multiplication of the corresponding continued fractions.

\subsubsection{Groups which preserve $\mathcal{F}$ and $\frac{1}{n}\mathcal{F}$}

As seen in Section~\ref{FT}, $Isom^+(\mathcal{F})=PSL_2(\mathbb{Z})$ is the maximal orientation-preserving group which preserves $\mathcal{F}$. Note that if $M$ is some orientation preserving matrix which preserves $\mathcal{F}$, then $\widetilde{M}:={ (n^*)^{-1}\circ{M}\circ{(n^*)}}$ is an orientation preserving matrix that preserves $\frac{1}{n}\mathcal{F}$. We can view this map $\widetilde{M}$ in the following way: First we use the $n^*$ map to scale $\frac{1}{n}\mathcal{F}$ to $\mathcal{F}$. Then we act upon $\mathcal{F}$ using the map $M$.  Finally, we scale $\mathcal{F}$ back to $\frac{1}{n}\mathcal{F}$ by using the map $(n^*)^{-1}$. 
Equivalently, if $L$ is a orientation preserving map which preserves $\frac{1}{n}\mathcal{F}$, then ${ (n^*)\circ{L}\circ{(n^*)^{-1}}}$ is an orientation preserving map which preserves $\mathcal{F}$. As a result, it follows that: $$Isom^+\left(\frac{1}{n}\mathcal{F}\right)=\{ (n^*)^{-1}\circ{A}\circ{(n^*)} : A\in{PSL_2(\mathbb{Z})}\}$$ is the maximal orientation preserving group which preserves $\frac{1}{n}\mathcal{F}$.

By explicit computation, we can see that these elements are of the following form: $${Isom^+\left(\frac{1}{n}\mathcal{F}\right)=\bigg\{\begin{psmallmatrix} a & \frac{b}{n}\\ nc & d \end{psmallmatrix} \in PSL_2(\mathbb{R}): \begin{psmallmatrix} a & b\\ c & d \end{psmallmatrix}\in{PSL_2(\mathbb{Z})} \bigg\}},$$ and $Isom^+(\frac{1}{n}\mathcal{F})$ takes on a natural group structure induced by $Isom^+(\mathcal{F})$.

We can recover a common subgroup of the maximal invariant subgroups of $\mathcal{F}$ and $\frac{1}{n}\mathcal{F}$ by taking the intersection of $Isom^+(\mathcal{F})$ and $Isom^+(\frac{1}{n}\mathcal{F})$. Again, by explicit computation, we see that $Isom^+(\mathcal{F})\cap{Isom^+(\frac{1}{n}\mathcal{F})}$ is given by:$$\Gamma_0(n):=\big\{\begin{psmallmatrix} a & b\\ c & d \end{psmallmatrix} \in PSL_2(\mathbb{Z}): c\equiv{0} \,\text{(mod n)} \big\}.$$ The group $\Gamma_0(n)$ is a subgroup of both $Isom^+(\mathcal{F})$ and $Isom^+(\frac{1}{n}\mathcal{F})$ by construction, and, therefore, preserves the structure of both $\mathcal{F}$ and $\frac{1}{n}\mathcal{F}$.

\subsubsection{The Structure of $\Gamma_0(n)\cdot{I}$}

When looking at $\Gamma_0(n)\cdot{I}$, the first thing to note is that the edge $I$ is in both $\mathcal{F}$ and $\frac{1}{n}\mathcal{F}$, for all $n\in\mathbb{N}$. Furthermore, $\Gamma_0(n)$ preserves both $\mathcal{F}$ and $\frac{1}{n}\mathcal{F}$ and, therefore, preserves their intersection $\mathcal{F}\cap\frac{1}{n}\mathcal{F}$. As a result, we can conclude that $\phi\cdot{I}$ is an edge of $\mathcal{F}\cap\frac{1}{n}\mathcal{F}$, for all $\phi\in\Gamma_0(n)$. This allows us to deduce that $\Gamma_0(n)\cdot{I}\subset\mathcal{F}\cap\frac{1}{n}\mathcal{F}$.

If $\phi:=\begin{psmallmatrix}a &b \\nc &d\end{psmallmatrix}\in\Gamma_0(n)\cdot{I}$, then $\phi\cdot{\infty}=\frac{a}{cn}$ and $\phi\cdot{0}=\frac{b}{d}$. Therefore, $\phi$ maps $I$ to an edge between $\frac{a}{cn}$ and $\frac{b}{d}$. Alternatively, if $\frac{a}{cn}$ and $\frac{b}{d}$ are neighbours in $\mathcal{F}$, then $|ad-bcn|=1$ and, therefore, either $\begin{psmallmatrix} a&b\\cn&d\end{psmallmatrix}$ is an element of $\Gamma_0(n)$ or $\begin{psmallmatrix} a&-b\\cn&-d\end{psmallmatrix}$ is an element of $\Gamma_0(n)$. This gives us the following lemma:

\begin{lem}\label{lem1}
Two points $A$ and $B$ are neighbours in $\Gamma_0(n)\cdot{I}$ if and only if they have reduced form $\frac{a}{nc}$ and $\frac{b}{d}$, with $|{ad-bnc}|=1$.
\end{lem}

 Comparing this to Lemma~\ref{lem2} allows us to deduce the following corollary:
\begin{cor}\label{subset}
The set of edges $\Gamma_0(n)\cdot{I}$ is a subset of $\mathcal{F}\cap\frac{1}{n}\mathcal{F}$. These sets are equivalent if and only if $n$ is a prime power.
\end{cor}

\begin{proof}
Recall from Lemma~\ref{lem2}, that two points $A$ and $B$ are neighbours in both $\mathcal{F}$ and $\frac{1}{n}\mathcal{F}$ if and only if they have reduced form $\frac{a}{n_{1}c_1}$ and $\frac{b}{n_2d_1}$, with $n=n_1 n_2$ and $|{an_2d_1-bn_1c_1}|=1$. Note that here we require that $\gcd(n_1,n_2)=1$, otherwise $|{an_2d_1-bn_1c_1}|=1$ has no solutions.

We will first show that if there is some $n_1,n_2\in\mathbb{N}$ such that $n=n_1n_2$ with $n_1>1$ and $n_2>1$, and  $\gcd(n_1,n_2)=1$, then $\Gamma_0(n)\cdot{I}\neq{\mathcal{F}\cap\frac{1}{n}\mathcal{F}}$. 

\noindent
\textbf{Claim:} Assume that there is some $n_1,n_2\in\mathbb{N}$ such that $n=n_1n_2$ with $n_1>1$ and $n_2>1$, and  $\gcd(n_1,n_2)=1$. Then there is at least one pair of points $\frac{a}{n_{1}c_1}$ and $\frac{b}{n_2d_1}$, which satisfy $|{an_2d_1-bn_1c_1}|=1$. 

\noindent
\textit{Proof of claim.}
First of all, let $a$ and $c_1$ be any numbers in $\mathbb{N}$, such that $\gcd(n_2,c_1)=1$. Then by the extended Euclidean algorithm - see~\cite{Jones:1998} - there are infinitely many solutions $(X,Y)$, to:
$$ an_2Y + n_1c_1X= 1.$$
Let $(X_1,Y_1)$ be one of these solutions. Then $\frac{a}{n_1c_1}$ and $\frac{-X_1}{n_2Y_1}$  satisfy:
$$an_2Y_1 - n_1c_1(-X_1)=an_2Y_1 + n_1c_1X_1=1.$$
Therefore, the points $\frac{a}{n_1c_1}$ and $\frac{-X_1}{n_2Y_1}$ satisfy $|{an_2Y_1-bn_1(-X_1)}|=1$. \hfill\textit{QED.}

In this case, $\frac{a}{n_1c_1}$ and $\frac{-X_1}{n_2Y_1}$ form an edge in $\mathcal{F}\cap\frac{1}{n}\mathcal{F}$, but not in $\Gamma_0(n)\cdot{I}$. In particular, $\Gamma_0(n)\cdot{I}\neq{\mathcal{F}\cap\frac{1}{n}\mathcal{F}}$.

Note that if $n$ is not a prime power, then by prime decomposition, we can always write $n=p_1^{\ell_1}n_2$ such that $p_1$ is prime, $p_1>1$ and $n_2>1$, and $\gcd(p_1,n_2)=1$. Therefore, in this case we can use the above argument to see that $\Gamma_0(n)\neq{\mathcal{F}\cap\frac{1}{n}\mathcal{F}}$.

On the other hand, if $n=p^\ell$ is some prime power, then we can only write $n=p_1^{\ell_1}n_2$ with $\gcd(p_1,n_2)=1$, if either $p_1=p$, $\ell_1=\ell$ and $n_2=1$, or $p_1=1$ and $n_2=p^\ell$. As a result, the only pairs of points satisfying Lemma~\ref{lem2} must be of the form $\frac{a}{p^\ell{c}}$ and $\frac{b}{d}$ with $\gcd(p^\ell,d)=1$. In particular, every edge which is in $\mathcal{F}\cap\frac{1}{n}\mathcal{F}$ must also be an edge $\Gamma_0(n)\cdot{I}$ by Lemma~\ref{lem1}.
\end{proof}

\subsubsection{Geodesics Intersecting $\Gamma_0(n)\cdot{I}$}

As we did for arbitrary edges in $\mathcal{F}\cap\frac{1}{n}\mathcal{F}$, if a geodesic $\zeta_\alpha$ intersects an edge $\phi\cdot{I}$ in $\phi\in\Gamma_0(n)\cdot{I}$, we can decompose $\zeta_\alpha$ into two paths: $\zeta_{\alpha,1}$, which runs along $\zeta_\alpha$ from $I$ to $\phi\cdot{I}$, and $\zeta_{\alpha,2}$, which runs along $\zeta_\alpha$ from $\phi\cdot{I}$ to $\alpha$. Since $\phi\cdot{I}$ is an edge in $\mathcal{F}\cap\frac{1}{n}\mathcal{F}$, we still have that:
$$(\zeta_\alpha,\mathcal{F})=(\zeta_{\alpha,1},\mathcal{F})\cdot(\zeta_{\alpha,2},\mathcal{F})$$ and 
$$(\zeta_\alpha,\frac{1}{n}\mathcal{F})=(\zeta_{\alpha,1},\frac{1}{n}\mathcal{F})\cdot(\zeta_{\alpha,2},\frac{1}{n}\mathcal{F}).$$

However, in this case we can gather even more information. Since $\zeta_{\alpha,2}$ is a geodesic ray, which starts at $\phi\cdot{I}$ and terminates at $\alpha$, it follows that $\phi^{-1}\cdot\zeta_{\alpha,2}$ is a geodesic ray, which starts at $I$ and terminates at $\beta:=\phi^{-1}\cdot{\alpha}$.
Since we assumed that $\zeta_\alpha$ non-trivially intersects $\phi\cdot{I}$, either $\beta<0$ or $\beta>0$. If $\beta>0$, then by Theorem~\ref{thma}, the cutting sequence $(\phi^{-1}\cdot\zeta_\alpha,\mathcal{F})$ is equivalent to the continued fraction expansion of $\beta$ and the cutting sequence $(\phi^{-1}\cdot\zeta_\alpha,\frac{1}{n}\mathcal{F})$ is equivalent to the continued fraction is equivalent to the continued fraction expansion of $n\beta$. Otherwise, if $\beta<0$, the cutting sequence $(\phi^{-1}\cdot\zeta_\alpha,\mathcal{F})$ is equivalent to the continued fraction expansion of $\frac{-1}{\beta}$ and the cutting sequence $(\phi^{-1}\cdot\zeta_\alpha,\frac{1}{n}\mathcal{F})$ is equivalent to the continued fraction is equivalent to the continued fraction expansion of $\frac{-1}{n\beta}$.
Since $\phi$ is an orientation preserving isomtery, the notion of a left or right triangle is preserved, and so, the geodesic ray $\zeta_{\alpha,2}$ will intersect $\mathcal{F}$ in the same way that $\phi^{-1}\cdot\zeta_{\alpha,2}$ intersects $\phi^{-1}\cdot\mathcal{F}$. However, since $\phi\in\Gamma_0(n)$ it follows that $\phi^{-1}\in\Gamma_0(n)$, and therefore, $\phi^{-1}\cdot\mathcal{F}=\mathcal{F}$. In particular, we have:
$$(\zeta_{\alpha,2},\mathcal{F})=(\phi^{-1}\cdot\zeta_{\alpha,2},\mathcal{F}).$$
By the same argument, it also follows that:
$$(\zeta_{\alpha,2},\frac{1}{n}\mathcal{F})=(\phi^{-1}\cdot\zeta_{\alpha,2},\frac{1}{n}\mathcal{F}).$$
This allows us to deduce that $\overline\beta$ is a tail of $\overline\alpha$ and $\overline{n\beta}$ is a tail of $\overline{n\alpha}$.

%
If a geodesic ray $\zeta_\alpha$ intersects $\phi\cdot{I}$, for some $\phi\in\Gamma_0(n)$, then we can look at how $\phi^{-1}\cdot{\zeta_\alpha}$ intersects $I$ relative to $\mathcal{F}$ and $\frac{1}{n}\mathcal{F}$ to recover information about the continued fraction expansion of $\alpha$ and $n\alpha$. The following proposition gives motivation for why we may want to do this:

\begin{prop}\label{pro2} If a continued fraction $\overline\alpha$ has a convergent denominator $q_k$, such that $q_k=nq_k'$, for some $n\in\mathbb{N}$ and some $q_k'\in\mathbb{N}_{>1}$, then $B(n\alpha)\geq{na_{k+1}}$. Furthermore, if $\frac{p_k}{q_k} =\frac{p_k}{nq'_k}$ is a convergent of $\overline\alpha$, then $\frac{p_k}{q'_k}$ is a convergent of $\overline{n\alpha}$.
\end{prop}

\begin{proof}
Since $q_k=nq_k'$, we can guarantee that $\gcd(q_{k-1},n)=1$. As a result, the edge between $\frac{p_k}{nq_k'}$ and $\frac{p_{k-1}}{q_{k-1}}$ is an edge in $\Gamma_0(n)\cdot{I}$. We can therefore find a map $\phi\in\Gamma_0(n)$ such that $\phi\cdot\infty=\frac{p_k}{nq_k'}$ and $\phi\cdot{0}=\frac{p_{k-1}}{q_{k-1}}$. Moreover, there is an edge in $\mathcal{F}$ between $\frac{p_k}{q_k}$ and each of the semi-convergents  $\frac{p_{\{k,m\}}}{q_{\{k,m\}}}$, by Corollary~\ref{semiconvedge}. We can therefore also guarantee that $\gcd(np_{k}',p_{k,m})=1$ and by extension, $\gcd(n,p_{k,m})=1$. In particular, not only are the edges between $\frac{p_k}{q_k}$ and $\frac{p_{\{k,m\}}}{q_{\{k,m\}}}$ in $\mathcal{F}$, but they are also edges in $\Gamma_0(n)\cdot{I}$.

Let $\zeta_{\alpha,2}$ be the geodesic ray which runs along $\zeta_\alpha$, starting at $\phi\cdot{I}$ and terminating at $\alpha$. Then $\phi^{-1}\cdot\zeta_{\alpha,2}$ starts at $I$ and terminates at $\phi^{-1}\cdot\alpha$. The map  $\phi^{-1}$ takes $\frac{p_k}{nq_k'}$ to the point at $\infty$, take $\frac{p_{k-1}}{q_{k-1}}$ to $0$, and preserves the structure of $\mathcal{F}$. Therefore, if $\frac{a}{c}$ is a neighbour of $\frac{p_k}{q_k}$, then $\phi^{-1}\cdot\frac{a}{c}$ must be a neighbour of $\infty$. In particular, each of the semi-convergents $\frac{p_{\{k,m\}}}{q_{\{k,m\}}}$ get mapped to a neighbour of $\infty$. 

The $\{k,1\}$-th semi-convergent $\frac{p_{\{k,1\}}}{q_{\{k,1\}}}$ is also a neighbour of $\frac{p_{k-1}}{q_{k-1}}$, and therefore $\phi^{-1}\cdot{\frac{p_{\{k,1\}}}{q_{\{k,1\}}}}$ must be a neighbour of both $0$ and $\infty$. There are two options, either $\phi^{-1}\cdot{\frac{p_{\{k,1\}}}{q_{\{k,1\}}}}=1$ or $\phi^{-1}\cdot{\frac{p_{\{k,1\}}}{q_{\{k,1\}}}}=-1$. By the same argument, $\phi^{-1}\cdot\frac{p_{\{k,2\}}}{q_{\{k,2\}}}$ is must be a neighbour of $\phi^{-1}\cdot{\frac{p_{\{k,1\}}}{q_{\{k,1\}}}}$ and $\infty$, \textit{i.e.} $\phi^{-1}\cdot\frac{p_{\{k,2\}}}{q_{\{k,2\}}}=2$, if $\phi^{-1}\cdot{\frac{p_{\{k,1\}}}{q_{\{k,1\}}}}=1,$ and $\phi^{-1}\cdot\frac{p_{\{k,2\}}}{q_{\{k,2\}}}=-2$, if $\phi^{-1}\cdot{\frac{p_{\{k,1\}}}{q_{\{k,1\}}}}=-1$. There are $a_{k+1}$ of these semi-convergents and so we can repeat this procedure to see that $\phi^{-1}\cdot{\zeta_\alpha}$ intersects $a_{k+1}$ edges which have $\infty$ as an endpoint. The other endpoint of these edges will either be $i$ or $-i$, for $i\in\{1,\ldots,a_{k+1}\}$ (depending on whether $\phi^{-1}\cdot{\frac{p_{\{k,1\}}}{q_{\{k,1\}}}}=1$ or $\phi^{-1}\cdot{\frac{p_{\{k,1\}}}{q_{\{k,1\}}}}=-1$). As such, the endpoint $\beta=\phi^{-1}\cdot\alpha$ of $\phi^{-1}\cdot\zeta_\alpha$ either satisfies $\beta>a_{k+1}$ or $\beta<-a_{k+1}$. 

If we take $a\in\mathbb{Z}$, then the points $a$, $a+1$ and $\infty$ form a triangle in $\mathcal{F}$. When we replace $\mathcal{F}$ with $\frac{1}{n}\mathcal{F}$, each of these triangles is effectively subdivided into $n$ triangles. This is because  for all $i\in\mathbb{Z}$, the points $\frac{i}{n}$, $\frac{i+1}{n}$ and $\infty$ form a triangle in $\frac{1}{n}\mathcal{F}$. As a result, we can guarantee that $\phi^{-1}\cdot\zeta_\alpha$ intersects at least $na_{k+1}$ triangles of this form in $\frac{1}{n}\mathcal{F}$. 

Since each of these triangles have $\infty$ as a fixed point, they form a fan in $\frac{1}{n}\mathcal{F}$. Therefore, $\phi\cdot{\infty}=\frac{p_k}{nq_k'}$ is the fixed point of the corresponding fan that $\zeta_\alpha$ forms with $\frac{1}{n}\mathcal{F}$. When we rescale by the $n^*$ map, this allows us to deduce that $n\cdot\frac{p_k}{nq_k'}=\frac{p_k}{q_k'}$ is a convergent of $n\alpha$. The fan corresponding to this convergent is of size at least $na_{k+1}$ by the above argument. Finally, since $q_k'\neq{0}$, we can guarantee that this fan is not the first fan in $n\alpha$. In particular, if $\overline{n\alpha};=[b_0;b_1,\ldots]$, then we have shown that there is some $b_\ell$ with $b_\ell\geq{}n_{a_k}$, for $\ell\geq{1}$. As a result, $B(n\alpha)\geq{b_\ell}\geq{na_{k+1}}$, as required.
\end{proof}

We should note that Proposition~\ref{pro2} is a folklore result in Diophantine approximation, and not terribly difficult to prove using basic knowledge of continued fractions. However, it does illustrate a fairly powerful technique that we will use later: using the structure of $\mathcal{F}$ and $\frac{1}{n}\mathcal{F}$ ``near'' $I$, to determine properties of geodesics which intersect $\Gamma_0(n)\cdot{I}$.  This motivates our definition of an \textit{infinite loop} mod $n$:

\begin{defn}[\textbf{a}]\label{infloopdef}
Let $\zeta_\alpha$ be a geodesic ray starting at the $y$-axis $I$ and terminating at the point $\alpha\in\mathbb{R}_{>0}$. Then $\zeta_\alpha$ is an \textit{infinite loop mod} $n$, if $\zeta_\alpha$ is disjoint from $\Gamma_0(n)\cdot{I}$  except for the edges of the form $I+k$, for $k\in\mathbb{Z}_{\geq{0}}$.
\end{defn}

If $n=p^\ell$, then by Corollary~\ref{subset}, a geodesic ray $\zeta_\alpha$ is an infinite loop mod $n$ if and only if $\zeta_\alpha$ is disjoint from $\mathcal{F}\cap\frac{1}{n}\mathcal{F}$ - except for the edges of the form $I+k$, for $k\in\mathbb{Z}_{\geq{0}}$. In this case, we see that the corresponding continued fraction expansions behave badly under integer multiplication - as discussed in Section~\ref{FnF}. However, if $n\neq{p^\ell}$, we may have that $\zeta_\alpha$ is an infinite loop mod $n$, but $\zeta_\alpha$ still intersects $\mathcal{F}\cap\frac{1}{n}\mathcal{F}$. In this case, the corresponding continued fraction may not behave particularly badly under integer multiplication, but it also does not behave particularly well, since we can not find a tail $\overline{\beta}$ of $\overline{\alpha}$ such that $\overline{n\beta}$ is also a tail of $\overline{n\alpha}$.

\subsubsection{An Alternative Definition of Infinite Loops}

As seen in Lemma~\ref{lem1}, two points $A$ and $B$ are neighbours in $\Gamma_0(n)\cdot{I}$  if and only if they have reduced form $\frac{a}{nc}$ and $\frac{b}{d}$, with $|{ad-bnc}|=1$. Viewing this information through the lens of infinite loops, we see that if $\zeta_\alpha$ is an infinite loop mod $n$, then $\zeta_\alpha$ can not intersect any edge in $\mathcal{F}$ which has an endpoint with denominator divisible by $n$ (except for the point at $\infty$). However, as seen in  Proposition~\ref{semiconvedge}, the semi-convergents of $\alpha$ are exactly the endpoints of the edges in $\mathcal{F}$ which $\zeta_\alpha$ intersects.
This leads to an equivalent definition of an infinite loop mod $n$ (as a real number).

\begin{defnum}[\ref{infloopdef} (b)] An \textit{infinite loop} mod $n$ is any real number $\alpha\in\mathbb{R}_{>0}$ with no semi-convergent denominators which are divisible by $n$ (other than $q_{-1}=0$). 
\end{defnum}

\begin{rem}
Here, we should note that if $\alpha\in\mathbb{Q}$, we will assume that the continued fraction expansion $\overline{\alpha}$ ends in a partial quotient of size $\infty$. The real number $\alpha$ still produces two separate continued fraction expansions of the form $[a_0;a_1,\ldots,a_m+1,\infty]$ and $[a_0;a_1,\ldots,a_m,1,\infty]$. The reason why we do this is because we may have rational numbers which are the endpoint of some edge in $\Gamma_0(n)\cdot{I}$, but do not have a semi-convergent denominator divisible by $n$, unless we include the final partial quotient of size $\infty$. Note that since:
$$\lim\limits_{k\to\infty}a_0+\cfrac{1}{a_1+\cfrac{1}{{\ldots}+\cfrac{1}{a_m+1+\cfrac{1}{k}}}}=a_0+\cfrac{1}{a_1+\cfrac{1}{{\ldots}+\cfrac{1}{a_m+1}}},$$
we will consider the continued fraction expansions $[a_0;a_1,\ldots,a_m+1,\infty]$ and $[a_0;a_1,\ldots,a_m+1]$ to be equivalent.
\end{rem}

Viewing infinite loops in terms of semi-convergents allows us to very easily deduce the following:

\begin{lem} If $\alpha$ is an infinite loop mod $n$, then $\alpha$ is an infinite loop mod $kn$, where $k\in\mathbb{N}$.
\end{lem}

\begin{proof}
Since $\alpha$ is an infinite loop mod $n$, it has no semi-convergent denominators which are divisible by $n$. By extension, $\alpha$ has no semi-convergent denominators divisible by $kn$, where $k\in\mathbb{N}$.
\end{proof}

\subsubsection{Existence of Infinite Loops mod $n$, for $n\geq{4}$}

In this section, we will show that for every $n\geq{4}$, there exist infinite loops mod $n$. In order to do this, we will first need to prove the following lemma:

\begin{lem}\label{infbound}
Let $\frac{a}{cn}$ and $\frac{b}{d}$ by two points in $\mathbb{R}_{>0}$ which satisfy $|ad-bcn|=1$. Then, for all $\alpha\in\mathbb{R}_{>0}$ satisfying $\min\left\{\frac{a}{cn},\frac{b}{d}\right\}\leq{\alpha}\leq\max\left\{\frac{a}{cn},\frac{b}{d}\right\}$, the geodesic ray $\zeta_\alpha$ is not an infinite loop mod $n$.
\end{lem}

\begin{proof}
Here, the edge between $\frac{a}{cn}$ and $\frac{b}{d}$ lies in $\Gamma_0(n)\cdot{I}$ by Lemma~\ref{lem1}. Furthermore, this edge separates $\mathbb{H}$ into two regions: one containing $I$, and the other containing $\alpha$. Since the geodesic ray $\zeta_\alpha$ runs from $I$ to $\alpha$, $\zeta_\alpha$ must necessarily intersect the edge between $\frac{a}{cn}$ and $\frac{b}{d}$. Therefore, $\zeta_\alpha$ can not be an infinite loop mod $n$.
\end{proof}

This allows us to prove the following:

\begin{prop} If $n\in\mathbb{N}$ and $n\geq{4}$, then there exist infinite loops mod $n$.
\end{prop}

\begin{proof}
In order to prove this statement, it is equivalent to show that there is no finite set of edges in $\Gamma_0(n)\cdot{I}$ connecting $0$ to $1$. Note that since $\begin{psmallmatrix} 1 &1 \\ 0 & 1 \end{psmallmatrix}$ is an element in $\Gamma_0(n)$ for all $n\in\mathbb{N}$, $\alpha\in[0,1]$ is an infinite loop mod $n$ if and only if $\alpha+k$ is an infinite loop for all $k\in\mathbb{Z}_{\geq{0}}$. As a result, it is sufficient to look for infinite loops in the interval $[0,1]$. 

As seen in Lemma~\ref{infbound}, if we have two points $\frac{a}{cn}$ and $\frac{b}{d}$ in the interval $[0,1]$ which satisfy $|ad-bcn|=1$, then for all $\alpha\in\mathbb{R}_{>0}$ satisfying $\min\left\{\frac{a}{cn},\frac{b}{d}\right\}\leq{\alpha}\leq\max\left\{\frac{a}{cn},\frac{b}{d}\right\}$, the geodesic ray $\zeta_\alpha$ is not an infinite loop mod $n$. If we assume that $\frac{a}{cn}<\frac{b}{d}$ and assume that there is another point of the form $\frac{e}{nf}>\frac{b}{d}$ with $|ed-bnf|=1$, then we can further conclude that for all $\alpha\in\mathbb{R}_{>0}$ satisfying $\frac{a}{cn}\leq{\alpha}\leq\frac{e}{nf}$, the corresponding geodesic rays $\zeta_\alpha$ are not infinite loops mod $n$. If there is a finite set of edges connecting $0$ to $1$, then we can use Lemma~\ref{infbound} on each of these edges to see that there is no infinite loop mod $n$, for all $0\leq{\alpha}\leq{1}$. However, if no such finite path exists, then there must be a non-empty set of points in $[0,1]$ which do not lie between any neighbours in $\Gamma_0(n)\cdot{I}$. If $\alpha$ is one of these points, then  the  corresponding geodesic ray $\zeta_\alpha$ does not intersect $\Gamma_0(n)\cdot{I}$. Therefore, $\alpha$ is an infinite loop mod $n$.

To find this set of edges, it is equivalent to find a finite sequence of rational points between $0$ and $1$ such that each consecutive pair of rational points are neighbours in $\Gamma_0(n)\cdot{I}$. This sequence of rational numbers will be of the form $\left\{\frac{0}{1}=\frac{b_0}{d_0},\frac{a_1}{c_1n},\frac{b_1}{d_1},\ldots,\frac{a_k}{c_kn},\frac{b_k}{d_k}=\frac{1}{1}\right\}$, where $\frac{b_{i-1}}{d_{i-1}}<\frac{a_i}{c_in}<\frac{b_{i}}{d_{i}}$, $a_i,b_i,c_i,d_i\in\mathbb{N}$ and  $\gcd(n,d_i)=1$. Given two points $A$ and $B$ and a sequence of rationals $\left\{A=A_0,A_1,A_2,\ldots,A_k=B\right\}$, we will say that this sequence is \textit{a sequence of neighbours in $\Gamma_0(n)\cdot{I}$ connecting $A$ and $B$} if $A_i<A_{i+1}$ and $A_i$ and $A_{i+1}$ are all neighbours in $\Gamma_0(n)\cdot{I}$ for all $i\in\left\{0,1,\ldots,k-1\right\}$. Similarly, if we have two points $A$ and $B$ and a sequence of rationals $\left\{A=A_0,A_1,A_2,\ldots,A_k=B\right\}$, we will say that this sequence is \textit{a sequence of neighbours in $\mathcal{F}$ connecting $A$ and $B$} if $A_i<A_{i+1}$ and $A_i$ and $A_{i+1}$ are all neighbours in $\mathcal{F}$ for all $i\in\left\{0,1,\ldots,k-1\right\}$. 

Since $\Gamma_0(n)\cdot{I}$ is a sub-graph of $\mathcal{F}\cap\frac{1}{n}\mathcal{F}$, which is in turn a sub-graph of $\mathcal{F}$, each edge $E$ in the finite set of edges in $\Gamma_0(n)\cdot{I}$ connecting $0$ to $1$, must also be an edge of $\mathcal{F}$. As a result, we will start with a sequence of neighbours in $\mathcal{F}$, and insert additional Farey neighbours to this sequence, until this sequence is also a sequence of neighbours in $\Gamma_0(n)\cdot{I}$. To show that this constructs a minimal sequence of neighbours in $\Gamma_0(n)\cdot{I}$ (should a minimal sequence exist), we will use the following claim:

\noindent
\textbf{Claim:} Assume that $\frac{a}{c},\frac{b}{d}\in\mathbb{Q}\cap[0,1]$ are neighbours in $\mathcal{F}$ with $\frac{a}{c}<\frac{b}{d}$. Then any sequence of neighbours in $\mathcal{F}$ of the form $\left\{\frac{a}{c}=\frac{a_0}{c_0},\frac{a_1}{c_1},\frac{a_2}{c_2},\ldots,\frac{a_k}{c_k}=\frac{b}{d}\right\}$ satisfying $\frac{a_{i-1}}{c_{i-1}}<\frac{a_i}{c_i}<\frac{a_{i+1}}{c_{i+1}}$ must either:
\begin{enumerate}

\item Only contain the points $\left\{\frac{a}{c},\frac{b}{d}\right\}$, or
\item Contain the point $\frac{a}{c}\oplus\frac{b}{d}=\frac{a+b}{c+d}$.
\end{enumerate}

\noindent
\textit{Proof of claim.}
Since $\frac{a}{c}$ and $\frac{b}{d}$ are  neighbours in $\mathcal{F}$ we know that there is an edge $E$ in $\mathcal{F}$ connecting these points. This edge separates the plane $\mathbb{H}$ into two regions: $E_+$, containing the interval $(\frac{a}{c},\frac{b}{d})$, and $E_-$, containing the intervals $[-\infty,\frac{a}{c})$ and $(\frac{b}{d},\infty]$. The sequence of neighbours $\left\{\frac{a}{c}=\frac{a_0}{c_0},\frac{a_1}{c_1},\frac{a_2}{c_2},\ldots,\frac{a_k}{c_k}=\frac{b}{d}\right\}$, must all lie in the interval $[\frac{a}{c},\frac{b}{d}]$, since we assumed that  $\frac{a_{i-1}}{c_{i-1}}<\frac{a_i}{c_i}<\frac{a_{i+1}}{c_{i+1}}$. In particular, the edges between each of these vertices must either be contained in $E_+$ or be the edge $E$ itself, \textit{i.e.} the sequence of neighbours in $\mathcal{F}$ is just $\left\{\frac{a}{c},\frac{b}{d}\right\}$.

If this is not the case, then we can assume that the sequence of neighbours in $\mathcal{F}$, given by $\left\{\frac{a}{c}=\frac{a_0}{c_0},\frac{a_1}{c_1},\frac{a_2}{c_2},\ldots,\frac{a_k}{c_k}=\frac{b}{d}\right\}$, contains a vertex $\frac{a_j}{b_j}$ which is not $\frac{a}{c},\frac{b}{d}$ or $\frac{a+c}{b+d}$.
Since $\frac{a}{c}\oplus\frac{b}{d}=\frac{a+b}{c+d}$ is a neighbour of both $\frac{a}{c}$ and $\frac{b}{d}$ in $\mathcal{F}$, the vertices $\frac{a+b}{c+d},\frac{a}{c}$ and $\frac{b}{d}$ form a triangle in $\mathcal{F}$. Furthermore, since $\frac{a_j}{c_j}\neq{\frac{a+c}{b+d}}$, the vertex $\frac{a_j}{c_j}$ can either lie in the interval $(\frac{a}{c},\frac{a+b}{c+d})$ or  $(\frac{a+b}{c+d},\frac{b}{d})$. We assume that $\frac{a_j}{c_j}$ lies in the interval $(\frac{a}{c},\frac{a+b}{c+d})$ - a similar argument can be made if $\frac{a_j}{c_j}$ lies in the interval $(\frac{a+b}{c+d},\frac{b}{d})$. Then, we take $E'$ to be the edge between $\frac{a}{c}$ and $\frac{a+b}{c+d}$, and assume $E'_+$ is the region containing the interval $(\frac{a}{c},\frac{a+b}{c+d})$. By assumption, the vertex $\frac{a_j}{c_j}$ is contained in the region $E'_+$. In the sequence $\left\{\frac{a}{c}=\frac{a_0}{c_0},\frac{a_1}{c_1},\frac{a_2}{c_2},\ldots,\frac{a_k}{c_k}=\frac{b}{d}\right\}$, there must be a subsequence of neighbours in $\mathcal{F}$ given by $\left\{\frac{a_j}{c_j},\frac{a_{j+1}}{c_{j+1}},\ldots,\frac{a_k}{c_k}=\frac{b}{d}\right\}$ which connects $\frac{a_j}{c_j}$ to $\frac{b}{d}$. However, $\frac{a_j}{c_j}$ lies in $E'_+$ and $\frac{b}{d}$ lies in $E'_-$. As a result, the corresponding sequence of edges in $\mathcal{F}$ must either contain the point $\frac{a+c}{b+d}$ or non-trivially intersect the edge $E'$. However, since $E'$ and the sequence of edges connecting $\frac{a_j}{c_j}$ to $\frac{b}{d}$ are all edges in the Farey tessellation, none of these edges can non-trivially intersect. Therefore, the subsequence of edges must pass through the point $\frac{a+c}{b+d}$ and so, the sequence of  neighbours $\left\{\frac{a_j}{c_j},\frac{a_{j+1}}{c_{j+1}},\ldots,\frac{a_k}{c_k}=\frac{b}{d}\right\}$ must contain the point $\frac{a+c}{b+d}$. Finally, since this subsequence contains the point $\frac{a+b}{c+d}$, so must our original sequence of neighbours $\left\{\frac{a}{c}=\frac{a_0}{c_0},\frac{a_1}{c_1},\frac{a_2}{c_2},\ldots,\frac{a_k}{c_k}=\frac{b}{d}\right\}$.\hfill \textit{QED}.\par

Given two Farey neighbours  $\frac{a}{c}$ and $\frac{b}{d}$ with $\frac{a}{c}<\frac{b}{d}$, we can use this claim to construct a minimal sequence of neighbours in $\Gamma_0(n)\cdot{I}$ between these points. We denote this minimal sequence $\widetilde{V}$. Firstly, we take the sequence of neighbours in $\mathcal{F}$ given by $V_0:=\left\{\frac{a}{c},\frac{b}{d}\right\}$ to be our initial sequence.
 If $\frac{a}{c}$ and $\frac{b}{d}$ are neighbours in $\Gamma_0(n)\cdot{I}$, then we will take $\widetilde{V}=V_0$, and we are done. Otherwise, by the above claim, the set $\widetilde{V}$ must include the point $\frac{a+b}{c+d}$. We know that $\frac{a+b}{c+d}$ is a Farey neighbour of both $\frac{a}{c}$ and $\frac{b}{d}$ and $\frac{a}{c}<\frac{a+b}{c+d}<\frac{b}{d}$. As a result, we can replace our initial sequence of neighbours $V_0=\left\{\frac{a}{c},\frac{b}{d}\right\}$ with the sequence of neighbours $V_1:=\left\{\frac{a}{c}, \frac{a+b}{c+d},\frac{b}{d}\right\}$. Since each consecutive pair of vertices in $V_1$ are neighbours in $\mathcal{F}$, we can consider each pair of vertices in the set $V_1$ individually and apply the same process on each of these pairs. For example, if $\frac{a}{c}$ and $ \frac{a+b}{c+d}$ are neighbours in $\Gamma_0(n)\cdot{I}$, then we do not need to construct any more vertices between them. However, if they are not neighbours in $\Gamma_0(n)\cdot{I}$, then our sequence of neighbours in $\Gamma_0(n)\cdot{I}$ must include their Farey neighbour $\frac{2a+b}{2c+d}$. As a result, we can replace the subsequence $\left\{\frac{a}{c},\frac{a+b}{c+d}\right\}$ by the subsequence $\left\{\frac{a}{c},\frac{2a+b}{2c+d},\frac{a+b}{c+d}\right\}$. We can then apply the same procedure on the subsequence $\left\{\frac{a+b}{c+d},\frac{b}{d}\right\}$ to form our next iterated set of neighbours in $\mathcal{F}$, which we denote $V_2$. We can then perform this procedure on each pair of vertices in $V_2$ to form a new set $V_3$, and then perform this procedure on the set $V_3$, and so on. Since we only add in additional neighbours between two points $A$ and $B$ when $A$ and $B$ are not neighbours in $\Gamma_0(n)\cdot{I}$, this process will form a minimal sequence of neighbours in $\Gamma_0(n)\cdot{I}$ between the points $A$ and $B$  - provided such a sequence of vertices exist. Starting with the initial set of vertices $V_0=\left\{0,1\right\}$, the process can be described algorithmically as follows:


\begin{enumerate}
\item Start with the set of vertices $V_0=\left\{\frac{0}{1},\frac{1}{1}\right\}$.
\item While $V_i$ is not of the required form, repeat the following process:
\begin{enumerate}
\item Take $V_{i+1}=\left\{\frac{0}{1}\right\}$.
\item For each pair of vertices $v_i$ and $v_{i+1}$ in $V_i$:

 If $v_i$ and $v_{i+1}$ are neighbours in $\Gamma_0(n)\cdot{I}$:
\begin{itemize}
\item Append $v_{i+1}$ onto $V_{i+1}$
\end{itemize}
Otherwise:
\begin{itemize}
\item Append $w=v_i\oplus{v_{i+1}}$ onto $V_{i+1}$.
\item Append $v_{i+1}$ onto $V_{i+1}$.
\end{itemize}
\end{enumerate}
\item End of algorithm.
\end{enumerate}

If we take $n=2$ then we have:
\begin{gather*}
V_0 =\left\{\frac{0}{1},\frac{1}{1}\right\},\\
V_1 = \left\{\frac{0}{1},\frac{1}{2},\frac{1}{1}\right\}.
\end{gather*}
At which point the process stops.

If we instead take $n=3$ then we have:
\begin{gather*}
V_0 =\left\{\frac{0}{1},\frac{1}{1}\right\},\\
V_1 = \left\{\frac{0}{1},\frac{1}{2},\frac{1}{1}\right\},\\
V_2 =\left\{\frac{0}{1},\frac{1}{3},\frac{1}{2},\frac{2}{3},\frac{1}{1}\right\}.
\end{gather*}
Again, the process stops at this point.

However, for $n=5$, we have:
\begin{gather*}
V_0 =\left\{\frac{0}{1},\frac{1}{1}\right\},\\
V_1 = \left\{\frac{0}{1},\frac{1}{2},\frac{1}{1}\right\},\\
V_2 =\left\{\frac{0}{1},\frac{1}{3},\frac{1}{2},\frac{2}{3},\frac{1}{1}\right\},\\
V_3 =\left\{\frac{0}{1},\frac{1}{4},\frac{1}{3},\frac{2}{5},\frac{1}{2},\frac{3}{5},\frac{2}{3},\frac{3}{4},\frac{1}{1}\right\},\\
V_4 =\left\{\frac{0}{1},\frac{1}{5},\frac{1}{4},\frac{2}{7},\frac{1}{3},\frac{2}{5},\frac{1}{2},\frac{3}{5},\frac{2}{3},\frac{5}{7},\frac{3}{4},\frac{4}{5},\frac{1}{1}\right\},\\
\ldots
\end{gather*}
Given two points $\frac{a}{c}$ and $\frac{b}{d}$, which are neighbours in $\mathcal{F}$, we can see from Lemma~\ref{lem1} that $\frac{a}{c}$ and $\frac{b}{d}$ are neighbours in $\Gamma_0(n)\cdot{I}$ if and only if exactly one of $c\equiv{0}\mod{n}$ or $d\equiv{0}\mod{n}$. Here, we should note that we can not have that both $c\equiv{0}\mod{n}$ and $d\equiv{0}\mod{n}$, since we know that $\gcd(c,d)=1$. 
In particular, assuming the points $\frac{a}{c}$ and $\frac{b}{d}$ are Farey neighbours, we only need to know the value of $c$ and $d$ mod $n$ to be able to tell if they are neighbours in $\Gamma_0(n)\cdot{I}$.  As a result, for us to construct a finite sequence of neighbours in $\Gamma_0(n)\cdot{I}$ of the form $\left\{\frac{0}{1}=\frac{b_0}{d_0},\frac{a_1}{c_1n},\frac{b_1}{d_1},\ldots,\frac{a_k}{c_kn},\frac{b_k}{d_k}=\frac{1}{1}\right\}$ it is a necessary condition that the sequence of denominators (taken $\mod{n}$) is of the form $\left\{\overline{d_0}=1,0,\overline{d_1},0,\ldots,0,\overline{d_k}=1\right\}$ where each $\overline{d_i}\in\left\{1,\ldots,n-1\right\}$. As a result, if we wish to show that the sequence of neighbours in $\Gamma_0(n)\cdot{I}$ of the form $\left\{\frac{0}{1}=\frac{b_0}{d_0},\frac{a_1}{c_1n},\frac{b_1}{d_1},\ldots,\frac{a_k}{c_kn},\frac{b_k}{d_k}=\frac{1}{1}\right\}$ does not exist, then it is sufficient to show that the corresponding sequence $\left\{\overline{d_0}=1,0,\overline{d_1},0,\ldots,0,\overline{d_k}=1\right\}$ does not exist.

If we start with two points $\frac{a}{c}$ and $\frac{b}{d}$ which are Farey neighbours, we can replace the sequence $V_0:=\left\{\frac{a}{c},\frac{b}{d}\right\}$ with the sequence $D_0:=\left\{\overline{c},\overline{d}\right\}$, where $\overline{c}\equiv{c}\mod{n}$, $\overline{d}\equiv{d}\mod{n}$ and $\overline{c},\overline{d}\in\left\{0,1,\ldots,n-1\right\}$. If one of $\overline{c}=0$ or $\overline{d}=0$, then we are done. Otherwise, $\frac{a}{c}$ and $\frac{b}{d}$ are not neighbours in $\Gamma_0(n)\cdot{I}$. In this case, we would replace the sequence $V_0:=\left\{\frac{a}{c},\frac{b}{d}\right\}$ with the sequence $V_1:=\left\{\frac{a}{c},\frac{a+b}{c+d},\frac{b}{d}\right\}$, and so we analogously replace the sequence $D_0:=\left\{\overline{c},\overline{d}\right\}$ with the sequence $D_1:=\left\{\overline{c},\overline{c+d},\overline{d}\right\}$, where $\overline{c+d}\equiv{c+d}\mod n$ and $\overline{c+d}\in\left\{0,1,\ldots,n-1\right\}$. If $\overline{c+d}=0$, then we are done. Otherwise, we can consider each consecutive pair in $D_1$ and perform the same procedure on each pair, \textit{i.e.} we perform the same procedure on $\left\{\overline{c},\overline{c+d}\right\}$ and $\left\{\overline{c+d},\overline{d}\right\}$. Iterating this procedure, we  can form a new algorithm to find a sequence of denominators of the required form $\left\{d_0=\overline{c},0,d_1,0,\ldots,0,d_k=\overline{d}\right\}$, where each $d_i\in\left\{1,\ldots,n-1\right\}$. For our initial set being $D_0:=\left\{1,1\right\}$ (corresponding to the set $V_0:=\left\{\frac{0}{1},\frac{1}{1}\right\}$),  the above procedure is described by the following algorithm.

\begin{enumerate}
\item Start with the set of denominators $D_0=\left\{1,1\right\}$.
\item While $D_i$ is not of the required form, repeat the following process:
\begin{enumerate}
\item Take $D_{i+1}=\left\{1\right\}$.
\item For each pair of denominators $d_i$ and $d_{i+1}$ in $D_i$:

 If $d_i$ and $d_{i+1}$ are neighbours in $\Gamma_0(n)\cdot{I}$:
\begin{itemize}
\item Append $d_{i+1}$ onto $D_{i+1}$
\end{itemize}
Otherwise:
\begin{itemize}
\item Append $e=d_i+{d_{i+1}}\mod{n}$ onto $D_{i+1}$.
\item Append $d_{i+1}$ onto $D_{i+1}$.
\end{itemize}
\end{enumerate}
\item End of algorithm.
\end{enumerate}
For example, for $n=5$ we would have:
\begin{gather*}
D_0 =\left\{{1},{1}\right\},\\
D_1 = \left\{{1},{2},{1}\right\},\\
D_2 =\left\{{1},{3},{2},{3},{1}\right\},\\
D_3 =\left\{{1},{4},{3},0,{2},0,{3},{4},{1}\right\},\\
D_4 =\left\{{1},0,{4},2,{3},0,{2},0,{3},2,4,0,{1}\right\},\\
D_5 =\left\{{1},0,{4},1,2,0,{3},0,{2},0,{3},0,2,1,4,0,{1}\right\},\\
\ldots
\end{gather*}
For $n>2$, we can always guarantee that the above process does not terminate  after the first iteration, and so, the above process creates the set $D_1=\left\{1,2,1\right\}$. Furthermore, for an arbitrary $n>3$, we can perform iterative Farey sums between the sub-sequence $\left\{1,2\right\}$ to obtain the sequence $\left\{1,0,n-1,n-2,\ldots,2\right\}$, and this sequence does not simply reduce to $\left\{1,0,2\right\}$, since  $n-1\neq{2}\mod n$ for $n>3$. If we perform the same process on the sub-sequence $\left\{n-1,n-2\right\}$ mod $n$, we obtain the sequence $\left\{n-1,0,1,2,\ldots,n-3,n-2\right\}$. Combining together these sequences, we see that iteratively performing the procedure on sub-sequence $\left\{1,2\right\}$ produces the sequence $\left\{1,0,n-1,0,1,2,\ldots,n-3,n-2,\ldots,2\right\}$. However, the sequence $\left\{1,0,n-1,0,1,2,\ldots,n-3,n-2,\ldots,2\right\}$  contains the sub-sequence $\left\{1,2\right\}$. This in turn implies that for $n>3$ we can not resolve any sub-sequence of the form $\left\{1,2\right\}$, since any attempt to do so produces another sub-sequence of the form $\left\{1,2\right\}$.  As a result, for $n>3$  we can not find a finite sequence of denominators: $$\left\{\overline{d_0}=1,0,\overline{d_1},0,\ldots,0,\overline{d_k}=1\right\}$$ corresponding to the finite sequence of neighbours in $\Gamma_0(n)\cdot{I}$ of the form: $$\left\{\frac{0}{1}=\frac{b_0}{d_0},\frac{a_1}{c_1n},\frac{b_1}{d_1},\ldots,\frac{a_k}{c_kn},\frac{b_k}{d_k}=\frac{1}{1}\right\}.$$ In particular, no such sequence of neighbours in $\Gamma_0(n)\cdot{I}$  can exist, for  $n>3$. Finally, this implies that there are infinite loops mod $n$ for all $n>3$.
\end{proof}

\subsection{\texorpdfstring{Infinite Loops and the $p$-adic Littlewood Conjecture}{Infinite Loops and the p-adic Littlewood Conjecture}}\label{ILpLC2}

We start this section by restating the $p$-adic Littlewood Conjecture and Corollary~\ref{uplow2}:

\begin{pconj}
For every real number $\alpha\in\mathbb{R}$, we have:
$$m_p(\alpha):=\inf\limits_{q\in\mathbb{N}}\left\{q\cdot|q|_p\cdot\|q\alpha\|\right\}=0.$$
\end{pconj}

\begin{corn}[\ref{uplow2}]
Every real number $\alpha\in\mathbb{R}\setminus\mathbb{Q}$ satisfies the following inequality:
$$\inf_{\ell\in\mathbb{N}}\frac{1}{B(p^\ell\alpha)+2}<m_{p}(\alpha)<\inf_{\ell\in\mathbb{N}}\frac{1}{B(p^\ell\alpha)}.$$
In particular, if $\alpha\in\mathbb{R}\setminus\mathbb{Q}$, then $\alpha$ satisfies pLC if and only if:
$$\sup\limits_{\ell\in\mathbb{N}\cup\{0\}}B(p^\ell\alpha)=\infty.$$
\end{corn}
 
As seen in the previous section, infinite loops mod $n$ behave ``badly'' when multiplied by $n$. In fact, infinite loops mod $n$ behave even worse when $n=p^\ell$. Since the $p$-adic Littlewood Conjecture is very closely related to the behaviour of the continued fractions expansions $\left\{\overline{p^m\alpha}:m\in\mathbb{N}\cup\left\{0\right\}\right\}$, it seems very natural that investigating  infinite loops mod $p^\ell$ may tell us something non-trivial about the $p$-adic Littlewood Conjecture. Our first confirmation of this fact, comes from the next lemma and its corollary.

This lemma can be viewed as a slightly weaker version of Proposition~\ref{pro2}. Instead of assuming $\alpha$ has a convergent denominator divisible by $n$, we assume that $\alpha$ is not an infinite loop mod $n$, \textit{i.e.} it has a semi-convergent denominator divisible by $n$. This lemma essentially states that if $\alpha$ is not an infinite loop mod $n$, then $B(\alpha)$ and $B(n\alpha)$ can not both be simultaneously small relative to $\sqrt{n}$.

\begin{lem}\label{noloop}
Assume that $\alpha\in\mathbb{R}_{>0}$ is not an infinite loop mod $n$. Then:
$$\max\left\{B(\alpha),B(n\alpha)\right\}\geq{\left\lfloor{2\sqrt{n}}\right\rfloor{-1}},$$

where $\lfloor\cdot\rfloor$ is the standard floor function.
\end{lem}

\begin{proof}
Assume $\alpha$ is not an infinite loop mod $n$ and let $\zeta_\alpha$ be the associated geodesic ray in $\mathbb{H}$. Since $\alpha$ is not an infinite loop mod $n$, there is an element $\phi\in\Gamma_0(n)$, where $\phi$ is not the identity, such that $\zeta_\alpha$ intersects the edge $\phi\cdot{I}$. We can apply the map $\phi^{-1}$ to the whole of $\mathbb{H}$ such that $\phi^{-1}\cdot\zeta_\alpha$ intersecting $\mathcal{F}$ resembles Fig.~\ref{nonloop} (a) - up to taking a mirror image in the $y$-axis. Taking a mirror image has no affect on this argument other than to swap the roles of left and right fans. As such, we shall assume that we are oriented as in the figure.

\begin{figure}[htbp]
\centering
\begin{subfigure}{.65\textwidth}
  \centering
  \includegraphics[width=1\linewidth]{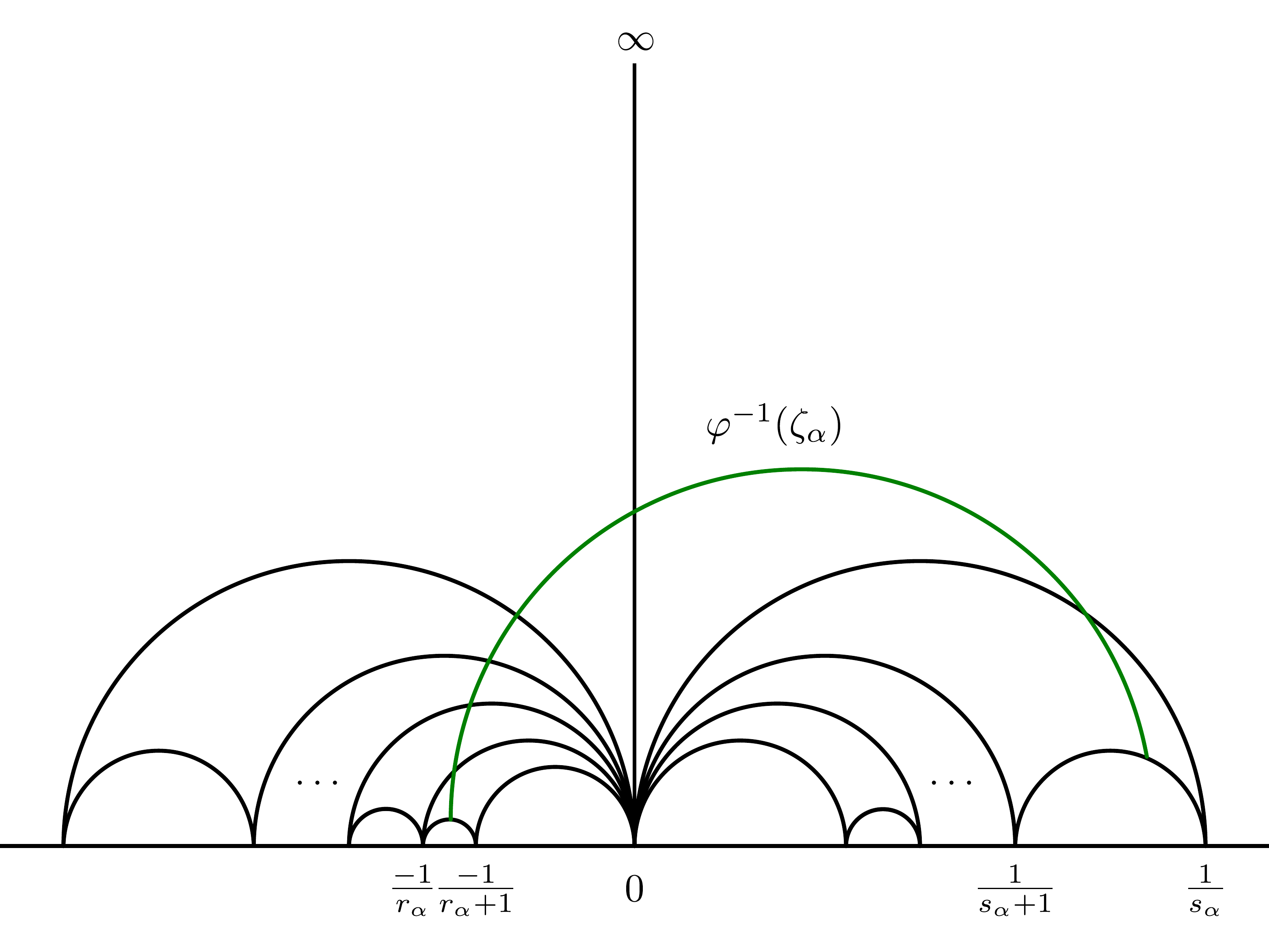}
  \caption{An example of a geodesic $\phi^{-1}\cdot\zeta_\alpha$ approaching $I$ by a right fan of size $r_\alpha$ in $\mathcal{F}$ and leaving $I$ via a fan of size $s_\alpha$. This results in a fan of size $(r_\alpha+s_\alpha)$.}
  \label{fig:sub1}
\end{subfigure}%

\begin{subfigure}{.65\textwidth}
  \centering
  \includegraphics[width=1\linewidth]{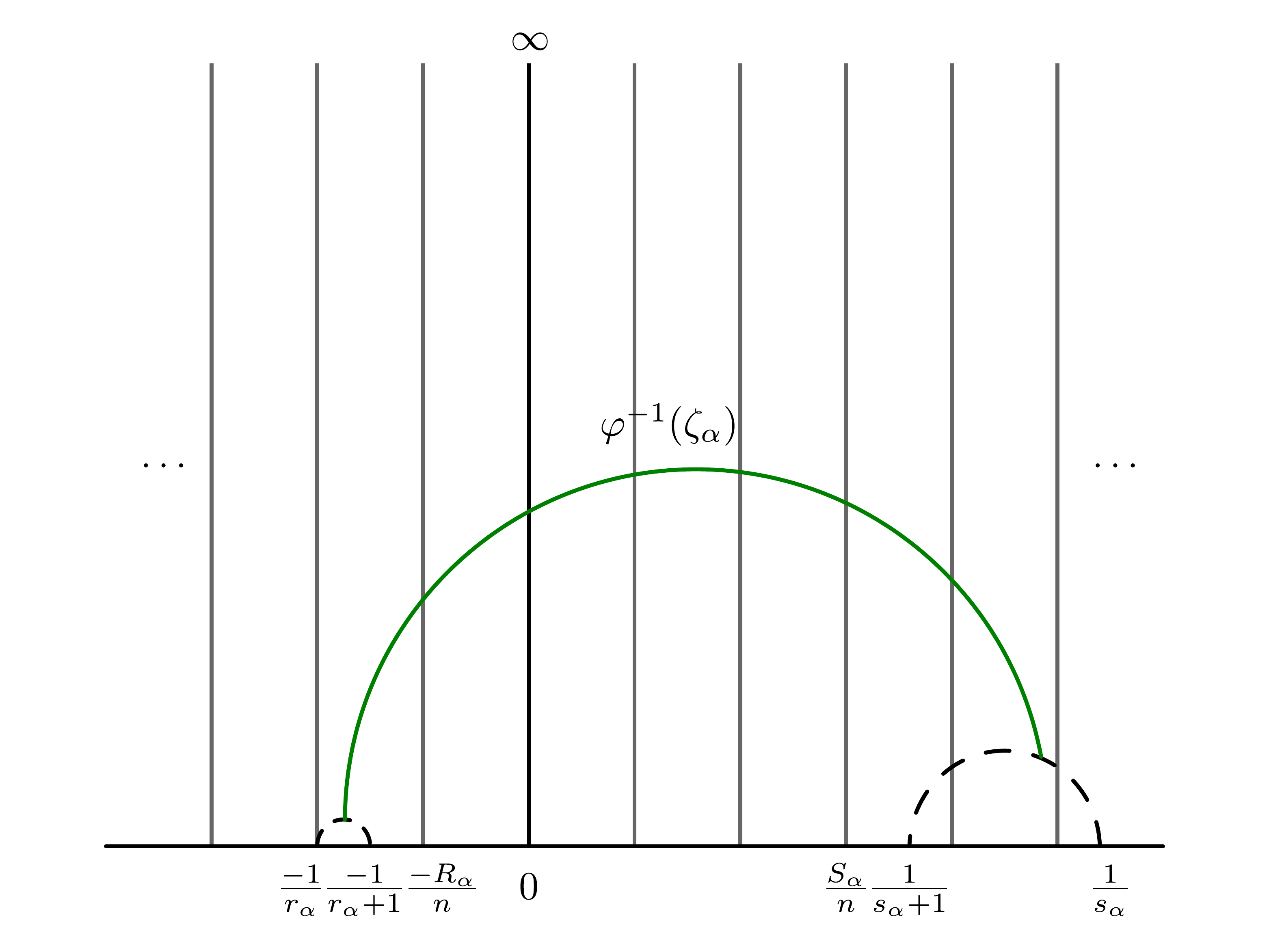}
  \caption{An example of how the geodesic $\phi^{-1}\cdot\zeta_\alpha$ intersects $\frac{1}{n}\mathcal{F}$. The lines between $\frac{-R_\alpha}{n}$ and $\frac{S_\alpha}{n}$ are necessarily intersected by $\phi^{-1}\cdot\zeta_\alpha$ and this results in a fan of size $\geq{(R_\alpha+S_\alpha)}$.}
  \label{fig:sub2}
\end{subfigure}
\caption{An example of a how geodesic ray $\zeta_\alpha$, which is not an infinite loop mod $n$,  intersects both $\mathcal{F}$, (a), and $\frac{1}{n}\mathcal{F}$, (b). This is considered up to re-framing by $\Gamma_0(n)$. }
\label{nonloop}
\end{figure}

We assume that the geodesic ray $\phi^{-1}\cdot\zeta_\alpha$ approaches the $y$-axis $I$ by a right fan of size $r_\alpha\in\mathbb{N}\cup\left\{0\right\}$ and leaves by a right fan of size $s_\alpha\in\mathbb{N}\cup\left\{0\right\}$. Here, we allow these fans to be of size $0$, however, in this case we interpret this fan to be a left fan. In this case, $\phi\cdot\zeta_\alpha$ either intersects $I$ and $I-1$ (when $r_\alpha=0$), or it intersects $I$ and $I+1$ (when $s_\alpha$=0). In either case, the point at infinity is a fixed point of this fan. This tells us that $\phi\cdot\infty$ is a convergent of $\alpha$. This point $\phi\cdot\infty$ will be of the form $\frac{a}{nc}$ and so by Proposition~\ref{pro2}, this induces $B(n\alpha)\geq{n}$ - in which case the result follows.

We therefore assume $r_\alpha,s_\alpha\geq{1}$ and note that $\phi^{-1}\cdot\zeta_\alpha$ approaches the $y$-axis from a value less than $-[0;r_\alpha,1]=\frac{-1}{r_\alpha+1}$.  Similarly, we can assume that $\phi^{-1}\cdot\zeta_\alpha$ departs the $y$-axis and approaches a point greater than $[0;s_\alpha,1]=\frac{1}{s_\alpha+1}$. Since $\frac{1}{n}\mathcal{F}$ has vertices between $\frac{i}{n}$ and $\infty$ for all $i\in\mathbb{N}$, we can ask how many of these lines the geodesic ray $\phi^{-1}\cdot\zeta_\alpha$ intersects in this neighbourhood. We see that there is some number $R_\alpha\in\mathbb{N}\cup\left\{0\right\}$ such that $\frac{-R_\alpha-1}{n}\leq{\frac{-1}{r_\alpha+1}\leq{\frac{R_\alpha}{n}}}$. One can then guarantee that $\phi^{-1}\cdot\zeta_\alpha$ intersects a left fan in $\frac{1}{n}\mathcal{F}$ of size at least $R_\alpha$  directly before approaching the $y$-axis. Note that here the value $R_\alpha$ can be defined as $R_\alpha:=\left\lfloor\frac{n}{r_\alpha+1}\right\rfloor$. By a similar process we can see that $\phi^{-1}\cdot\zeta_\alpha$ intersects a left fan in $\frac{1}{n}\mathcal{F}$ of size at least $S_\alpha:=\left\lfloor\frac{n}{s_\alpha+1}\right\rfloor$ in $\frac{1}{n}\mathcal{F}$ directly after leaving the $y$-axis, see Fig.~\ref{nonloop} (b). These fans concatenate to form a fan of size $R_\alpha+S_\alpha$ in $\frac{1}{n}\mathcal{F}$. Therefore, we know that $\overline{\alpha}$ has a term of size at least $r_\alpha+s_\alpha$ and $\overline{n\alpha}$
has a term of size at least $R_\alpha+S_\alpha$. We conclude that $B(\alpha)\geq{r_\alpha+s_\alpha}$ and $B(n\alpha)\geq{R_\alpha+S_\alpha}$, and by extension $\text{max}\left\{B(\alpha),B(n\alpha)\right\}\geq{\text{max}\left\{r_\alpha+s_\alpha,\,R_\alpha+S_\alpha\right\}}$.\\

 We assume that $r_\alpha+s_\alpha\leq\left\lfloor{2\sqrt{n}}\right\rfloor-2$, since otherwise we would have $B(\alpha)\geq{}\left\lfloor{2\sqrt{n}}\right\rfloor-1$.
If we fix $0\leq{r_\alpha}\leq\left\lfloor{2\sqrt{n}}\right\rfloor-2$, then  $0\leq{s_\alpha}\leq\left\lfloor{2\sqrt{n}}\right\rfloor-2-r_\alpha$. For all $s_\alpha$ in this range, we note:
$$ S_\alpha=\left\lfloor\frac{n}{s_\alpha+1}\right\rfloor\geq\left\lfloor\frac{n}{\left\lfloor{2\sqrt{n}}\right\rfloor-2-r_\alpha+1}\right\rfloor=\left\lfloor\frac{n}{\left\lfloor{2\sqrt{n}}\right\rfloor-r_\alpha-1
}\right\rfloor,$$

and:
\begin{align*}
R_\alpha+S_\alpha=\left\lfloor\frac{n}{r_\alpha+1}\right\rfloor+\left\lfloor\frac{n}{s_\alpha+1}\right\rfloor\geq{}&\left\lfloor\frac{n}{r_\alpha+1}\right\rfloor+\left\lfloor\frac{n}{\left\lfloor{2\sqrt{n}}\right\rfloor-r_\alpha-1}\right\rfloor\\
\geq{}&\left\lfloor\frac{n}{r_\alpha+1}+\frac{n}{\left\lfloor{2\sqrt{n}}\right\rfloor-r_\alpha-1}\right\rfloor-1.
\end{align*}

We can find a lower bound estimation for this by considering the following equation:
$$ \text{f(x)}:=\frac{n}{x+1}+\frac{n}{\left\lfloor{2\sqrt{n}}\right\rfloor-x-1}, \text{ for x} \in \left[0,\lfloor{2\sqrt{n}}\rfloor-2\right], $$

and noting that $\left\lfloor{\text{f(x)}}\right\rfloor$ is minimised when $\text{f(x)}$ is minimised. The derivative of $\text{f(x)}$ is given by:
$$ \text{f}'\text{(x)}=n\left(\frac{-1}{(x+1)^2}+\frac{1}{(\left\lfloor{2\sqrt{n}}\right\rfloor-x-1)^2} \right). $$

\noindent
Note that we can write $\left\lfloor{2\sqrt{n}}\right\rfloor=2\left\lfloor{\sqrt{n}}\right\rfloor+\delta$ where $\delta=0,1$.

\noindent
\textbf{1.  Assume $\delta=0$:}

In this case, $\text{f}'\text{(x)}=0$ if and only if $x=\left\lfloor{\sqrt{n}}\right\rfloor-1$, and so, $x=\left\lfloor{\sqrt{n}}\right\rfloor-1$ must be either a minima or a maxima (since $\text{f(x)}$ is symmetric in x).  At $x=\left\lfloor{\sqrt{n}}\right\rfloor-1$, we have:
\begin{align*}
\lfloor{\text{f(x)}}\rfloor={}&\left\lfloor\frac{n}{\left\lfloor{\sqrt{n}}\right\rfloor}+\frac{n}{\left\lfloor{\sqrt{n}}\right\rfloor}\right\rfloor-1 \\ ={}& \left\lfloor\frac{2n}{\left\lfloor{\sqrt{n}}\right\rfloor}\right\rfloor-1\\
\geq{}&  \left\lfloor\frac{2n}{\sqrt{n}}\right\rfloor-1\\
={}& \left\lfloor{2\sqrt{n}}\right\rfloor-1
\end{align*}

We note that $\left\lfloor\text{f(0)}\right\rfloor\geq{n}$, which is greater than or equal to $2\left\lfloor{\sqrt{n}}\right\rfloor-1$ for all $n\in\mathbb{N}$. Therefore, $\lfloor{\text{f(}\left\lfloor{\sqrt{n}}\right\rfloor-1\text{)}}\rfloor$ is a minima.

\noindent
\textbf{2.  Assume $\delta=1$:}

In this case, $\text{f}'\text{(x)}=0$ if and only if $x=\left\lfloor{\sqrt{n}}\right\rfloor-\frac{1}{2}$. At $x=\left\lfloor{\sqrt{n}}\right\rfloor-\frac{1}{2}$, we have:
\begin{align*}
\lfloor{\text{f(x)}}\rfloor={}&\left\lfloor\frac{n}{\left\lfloor{\sqrt{n}}\right\rfloor+\frac{1}{2}}+\frac{n}{\left\lfloor{\sqrt{n}}\right\rfloor+\frac{1}{2}}\right\rfloor-1\\
 ={}& \left\lfloor\frac{4n}{2\left\lfloor{\sqrt{n}}\right\rfloor+1}\right\rfloor-1\\
={}& \left\lfloor\frac{4n}{\left\lfloor{2\sqrt{n}}\right\rfloor}\right\rfloor-1\\
\geq{}&  \left\lfloor\frac{4n}{2\sqrt{n}}\right\rfloor-1\\
={}& \left\lfloor{2\sqrt{n}}\right\rfloor-1
\end{align*}

We note that $\left\lfloor\text{f(0)}\right\rfloor\geq{n}$, which is greater than or equal to $2\left\lfloor{\sqrt{n}}\right\rfloor-1$ for all $n\in\mathbb{N}$. Therefore, $\lfloor{\text{f(}\left\lfloor{\sqrt{n}}\right\rfloor-\frac{1}{2}\text{)}}\rfloor$ is a minima.

Therefore, for all $r_\alpha+s_\alpha\leq{\left\lfloor{2\sqrt{n}}\right\rfloor}-2$, we have that: $$R_\alpha+S_\alpha\geq\min\limits_{x\in\mathbb{R}_{\geq{}0}}\left\lfloor\text{f(x)}\right\rfloor\geq\left\lfloor{2\sqrt{n}}\right\rfloor-1.$$ Finally, it follows that $\text{max}\left\{r_\alpha+s_\alpha,\,R_\alpha+S_\alpha\right\}\geq{\left\lfloor{2\sqrt{n}}\right\rfloor}-1$ for all possible $r_\alpha$ and $s_\alpha$.
\end{proof}

The above lemma gives us a lower bound for $\max\left\{B(\alpha),B(n\alpha)\right\}$ if $\alpha$ is not an infinite loop mod $n$. Recall from Corollary~\ref{uplow2}, that for all $\alpha\in\mathbb{R}$ we have:
 $$\inf\limits_{\ell\in\mathbb{N}\cup\left\{0\right\}}\left\{\frac{1}{B(p^\ell\alpha)+2}\right\}\leq{}m_p(\alpha)\leq\inf\limits_{\ell\in\mathbb{N}\cup\left\{0\right\}}\left\{\frac{1}{B(p^\ell\alpha)}\right\}.$$ This leads to the following corollary:

\begin{cor}\label{infl} If $\alpha\in\mathbb{R}_{>0}$ is not an infinite loop mod $p^m$, then:
\[ m_p(\alpha)\leq{\frac{1}{\left\lfloor{2\sqrt{p^m}}\right\rfloor{-1}}}.\] 
\end{cor}

\begin{proof} Since we know that:
$$m_p(\alpha)\leq{}\inf\limits_{\ell\in\mathbb{N}\cup\left\{0\right\}}\left\{\frac{1}{B(p^\ell\alpha)}\right\},$$
we can conclude that $$m_p(\alpha)\leq\frac{1}{B(p^j\alpha)},$$ for any $j\in\mathbb{N}\cup\left\{0\right\}$. Since $\alpha$ is not an infinite loop mod $p^m$, we know by Lemma~\ref{noloop}, that: $$\max\left\{B(\alpha),B(p^m\alpha)\right\}\geq{\left\lfloor{2\sqrt{p^m}}\right\rfloor{-1}}.$$
Combining this information together, we see that:
$$m_p(\alpha)\leq\min\left\{\frac{1}{B(\alpha)},\frac{1}{B(p^m\alpha)}\right\}\leq{\frac{1}{\left\lfloor{2\sqrt{p^m}}\right\rfloor{-1}}},$$
as required.
\end{proof}

\begin{cor}\label{c}
Let $\alpha\in\textbf{Bad}$ and assume  there is some sequence of natural numbers $\left\{\ell_m\right\}_{m\in\mathbb{N}}$  such that $p^{\ell_m}\alpha$ is not an infinite loop mod $p^m$. Then $\alpha$ satisfies pLC.
\end{cor}


\begin{proof}
From Corollary~\ref{infl}, we can conclude that for any $\alpha\in\mathbb{R}_{>0}$, if there is a sequence of natural numbers $\{\ell_m)\}_{m\in\mathbb{N}}$ such that $p^{\ell_m}\alpha$ is not an infinite loop mod $p^m$, then we have:  $$m_p(\alpha)\leq\lim\limits_{m\to\infty}\frac{1}{B(p^{\ell_m}\alpha)}\leq{}\lim\limits_{m\to\infty}{\frac{1}{\left\lfloor{2\sqrt{p^m}}\right\rfloor{-1}}}=0.$$ Therefore, $\alpha$ satisfies pLC. 
\end{proof}
Here, we should note that the sequence $\left\{\ell_m\right\}$ need not be monotonically increasing. For example, we may have that $\alpha$ is an infinite loop mod $p$, but there exist some $K\in\mathbb{N}$, such that $p^K\alpha$ is not an infinite loop mod $p^m$ for all $m\in\mathbb{N}$. In this case, the sequence $l_m=K$ for all $m\in\mathbb{N}$ would allow us to show that $\alpha$ satisfies $p$LC. We omit the proof here, but  such a constant exists for all real numbers with an eventually recurrent continued fraction expansion - by extension this is also true for all real numbers which admit an eventually periodic continued fraction expansion. See \cite{paper,thesis}. In particular, for every real number $\alpha$ and  every prime $p$, there exists some integer $K\in\mathbb{N}$ such that $p^K\alpha$ admits a semi-convergent denominator divisible by $p^m$, for every possible $m\in\mathbb{N}$.

In contrast to Corollary~\ref{c}, if there exists an $m\in\mathbb{N}$ such that $p^\ell\alpha$ is an infinite loop mod $p^m$, for all $\ell\in\mathbb{N}\cup\left\{0\right\}$,  then $\alpha$ is a counterexample to pLC.

\begin{lem}\label{count}
Let $\alpha\in\textbf{Bad}$ and assume there exists an $m\in\mathbb{N}$ such that $p^\ell\alpha$ is an infinite loop mod $p^m$, for all $\ell\in\mathbb{N}\cup\left\{0\right\}$. Then $\alpha$ is a counterexample to pLC and $m_p(\alpha)\geq{\frac{1}{p^m-2}}$.
\end{lem}

In order to prove this statement, we will first prove the following claim:

\noindent
\textbf{Claim:} If $\beta\in\mathbb{R}_{>0}$ is a real number  such that $p^j\beta$ is an infinite loop mod $p^m$ for all $j\in\mathbb{N}\cup\left\{0\right\}$, then $b_{k+1}\leq{p^m-4}$ for all $k\in\mathbb{N}\cup\{0\}$, where $\overline{\beta}=[b_0;b_1,\ldots]$. In particular, we can then conclude that $B(\beta)\leq{p^m-4}$.

\begin{proof}[Proof of claim]
Let $b_{k+1}$ be an arbitrary partial quotient of $\beta$ for some $k\in\mathbb{N}\cup\{0\}$ and consider the following two cases for the corresponding convergent denominator $q_k$:\\

\hspace{10mm}\textbf{(Case I):} The prime $p$ and $q_{k}$ are coprime.\\

\hspace{10mm}\textbf{(Case II):} The prime $p$ and $q_{k}$ are not coprime.\\

\noindent
\textbf{(Case I):} Since $q_{k}$ is coprime with $p$, we know that there are infinitely many neighbours of $\frac{p_{k}}{q_{k}}$ in $\mathcal{F}$ which have a denominator divisible by $p^m$. This is analogous to the fact the $0$ has infinitely many neighbours of the form $\frac{1}{p^mj}$ in $\mathcal{F}$, where $j\in\mathbb{N}$. The corresponding geodesic ray $\zeta_{\beta}$ must not intersect any of the geodesic arcs from $\frac{p_{k}}{q_{k}}$ to these neighbours. As a result, there is a unique pair of neighbours in $\mathcal{F}$, $\frac{a_1}{c_1p^m}$ and $\frac{a_2}{c_2p^m}$, such that the arcs between these points and $\frac{p^\ell_{k}}{q^\ell_{k}}$ separate $\zeta_{\beta}$ from all other neighbours of $\frac{p_{k}}{q_{k}}$ in $\mathcal{F}$ whose denominator divisible by $p^m$. See Figure~\ref{counter}.

\begin{figure}[htb]
\centering
  \includegraphics[width=0.8\linewidth]{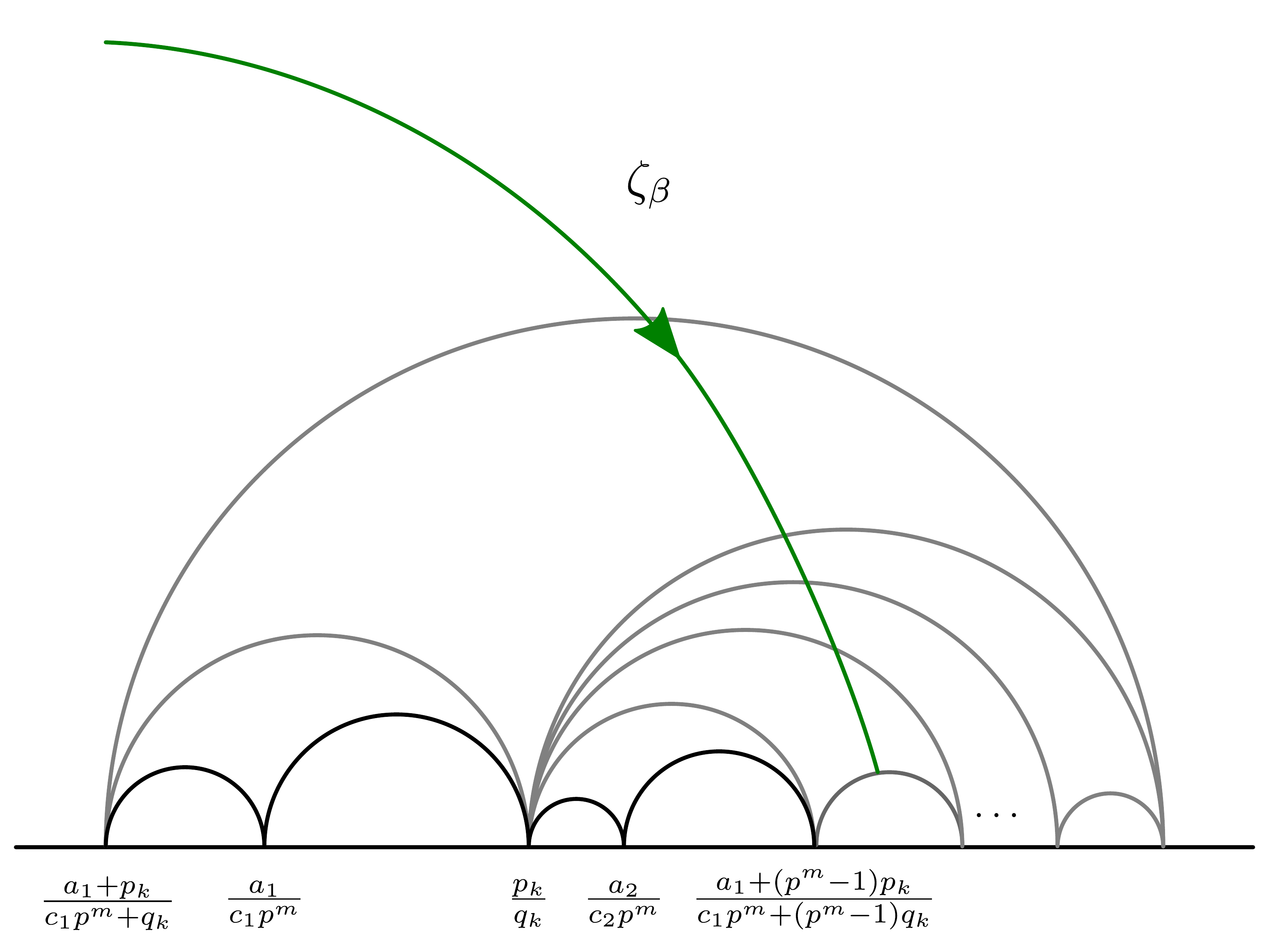} 
  \caption{An image of $\zeta_{\beta}
$ cutting a fan (relative to $\mathcal{F}$) with fixed point $\frac{p_k}{q_k}$ and $\gcd(p,q_k)=1$. In this scenario, $\zeta_{\beta}$ is an infinite loop mod $p^m$ and forms the largest fan possible for this $b_{k+1}$.}  \label{counter}
\end{figure}

Similarly, we can express all other neighbours of $\frac{p_{k}}{q_{k}}$ in $\mathcal{F}$ in this region by using either Farey addition or Farey subtraction  on $\frac{p_{k}}{q_{k}}$ and $\frac{a_1}{c_1p^m}$ (up to relabelling). In the case that we have to do Farey subtraction, we can replace the representation of $A=\frac{p_k}{q_k}$ with $A=\frac{-p_k}{-q_k}$ and do Farey addition instead. In either case, we can express the neighbours in this region as: 
\[
n_i=\frac{a_1+i\cdot{p^\ell_{k}}}{c_1p^m + i\cdot{q^\ell_{k}}}
,\]  where  $i\in\left\{0,1,\ldots,p^m\right\}$ \text{ and } $n_0=\frac{a_1}{c_1p^m}$ \text{ and } $n_{p^m}=\frac{a_2}{c_2p^m}$.

Two of these neighbours will be fixed vertices for the previous and subsequent fans, and we label these neighbours as $n_s$ and $n_t$ with $t>s$. The size of the fan $b_{k+1}$ is given by $t-s$. The points $\frac{p_{k}}{q_{k}},n_{s}$ and $n_{s-1}$ form a triangle in $\mathcal{F}$, and so, since $n_s$ is a convergent denominator of $\zeta_{\beta}$, the point $n_{s-1}$ must be a semi-convergent of $p^\ell\alpha$. Similarly, since $n_t$ is  the convergent of the next partial quotient, the point $n_{t+1}$ is a semi-convergent of $\beta$. If either $n_0$ or $n_{p^m}$ are semi-convergents of $\beta$, then, since they are of the form $\frac{A}{Cp^m}$ with $C\neq{0}$, we can conclude that $\beta$ is not an infinite loop mod $p^m$. It follows that for $\zeta_\beta$ to be an infinite loop mod $p^m$, we have $s\in\left\{2,\ldots,p^m-3\right\}$ and $t\in\left\{3,\ldots,p^m-2\right\}$. Therefore, the maximum size of the fan is $b_{k+1}$ is given by $(\max{t}-\min{s})=p^m-2-2=p^m-4$, as required.  \hfill\textit{QED.}

%
%
%

\noindent
\textbf{(Case II):} In this case, there is some $j\in\mathbb{N}$ such that $p^j\mid{q_{k}}$ and $p^{j+1}\nmid{q_{k}}$. We will write $q_k=p^jd_k$, where $d_k\in\mathbb{N}$ and $\gcd(q'_k,p)=1$.  Therefore, by Proposition~\ref{pro2}, we can deduce that $B(p^j\beta)\geq{}p^jb_{k+1}$ and ${p_k}{d_k}$ is a convergent of $p^j\beta$. We wish to show that if $b^j_{k'+1}$ is the partial quotient of $p^j\beta$ corresponding to the convergent ${p_k}{d_k}$, then we have ${b_{k+1}\cdot{p^j}}\leq{b^{j}_{k'+1}}$. Since $\gcd(d_k,p)=1$, we can use Case I to conclude that $b_k<b^j_{k'+1}\leq{p^m-4}$.

The geodesic ray $\zeta_{\beta}$ forms a fan $B_{k+1}$  with $\mathcal{F}$ of size $b_{k+1}$ and this fan has a fixed vertex $\frac{p_{k}}{q_{k}}$. Since $p^j\mid{q_{k}}$, any neighbour $\frac{a}{c}$ of $\frac{p_k}{q_k}$ in $\mathcal{F}$ must satisfy $\gcd(c,p^j)=1$. By Lemma~\ref{lem1}, the edge between a neighbour of this form and $\frac{p_k}{q_k}$ must be an edge of $\mathcal{F}\cap{\frac{1}{p^j}\mathcal{F}}$. Therefore, since every edge that $\zeta_\beta$ intersects in $B_{k+1}$ has $\frac{p_{k}}{q_{k}}$ as one of its endpoints, we can conclude that each of these edges lie in $\mathcal{F}\cap{\frac{1}{p^j}\mathcal{F}}$. Since these edges lie in $\mathcal{F}\cap{\frac{1}{p^j}\mathcal{F}}$, we can guarantee that $\frac{p_k}{q_k}$ is a fixed point of some fan $B^j_{k'+1}$ in the cutting sequence $(\zeta_\beta,\frac{1}{p^j}\mathcal{F})$. After having corrected for scaling, we can observe that $\frac{p_{k}}{d_{k}}$ is a convergent of $p^j\beta$, where $q_{k}=d_{k}\cdot{p^j}$, as above. Each triangle in $B_k$ is sub-divided into $p^j$ triangles when we replace $\mathcal{F}$ with $\frac{1}{p^j}\mathcal{F}$, as described in Proposition~\ref{pro2}. Therefore, if $b^{j}_{k'+1}$ is the partial quotient of $p^j\beta$ corresponding to the fan $B^j_{k'+1}$ - with corresponding convergent $\frac{p_{k}}{d_{k}}$ - then $b^{j}_{k'+1}$ satisfies: $${{p^j}\cdot{}b_{k+1}}\leq{}b^{j}_{k'+1}.$$ Since $gcd(d_{k},p)=1$, we can use Case I to see that $b^{j}_{k'+1}<p^m-4$. However, since ${{p^j}\cdot{}b_{k+1}}\leq{b^{j}_{k'+1}}$, we can conclude ${b_{k+1}}\leq{p^m-4}$, as required.\hfill\textit{QED.}

Finally, since the partial quotients $b_{k+1}$ of $\beta$ are all bounded above by $p^m-4$, we can conclude that $B(\beta)\leq{p^m-4}$ and this completes the proof of the claim.
\end{proof}

\begin{proof}[Proof of Lemma~\ref{count}]
For each $\ell\in\mathbb{N}\cup\{0\}$, $p^\ell\alpha$ is not an infinite loop mod $p^m$. Therefore $p^{\ell+j}\alpha$ is also an infinite loop mod $p^m$ for all $j\in\mathbb{N}\cup\left\{0\right\}$. As a result, we can replace $\beta$ in the above claim by $p^\ell\alpha$, and we see that $B(p^\ell\alpha)\leq{p^m-4}$, for all $\ell\in\mathbb{N}\cup\left\{0\right\}$.  As seen in Corollary~\ref{uplow2}, we know that: $$m_p(\alpha)\geq{}\inf\limits_{\ell\in\mathbb{N}\cup\left\{0\right\}}\frac{1}{B(p^\ell\alpha)+2}.$$ Finally, we can conclude that: $$m_p(\alpha)\geq{}{\frac{1}{p^m-2}}.$$
\end{proof}
Combining Corollary~\ref{infl} and Lemma~\ref{count} gives us the main theorem for this paper.

\begin{thm}\label{theorem}
Let $\alpha\in\textbf{Bad}$. Then
$\alpha$ satisfies pLC if and only if there is a sequence of natural numbers $\left\{\ell_m\right\}_{m\in\mathbb{N}}$ such that $p^{\ell_m}\alpha$ is not an infinite loop mod $p^m$.
\end{thm}

\subsection*{Funding and Acknowledgements}
This research was supported by the Engineering and Physical Sciences Research Council (EPSRC) [grant no. EP/N509462/1], awarded through Durham University.\\

\noindent
I would like to thank my Ph.D. supervisor, Dr. Anna Felikson, for her continued support and encouragement throughout the project. I would also like to thank Matthew Northey for his helpful discussions surrounding this topic.

\nocite{*}
\printbibliography
\end{document}